\documentclass{article}
\usepackage[utf8]{inputenc}

\title{Groups elementarily equivalent to the classical matrix groups}

\author{Alexei Myasnikov\footnote{Address: Stevens Institute of Technology, Department of Mathematical Sciences, Hoboken, NJ 07087, USA. Email: amiasnik@stevens.edu.} , 
 Mahmood Sohrabi\footnote{Address: Stevens Institute of Technology, Department of Mathematical Sciences, Hoboken, NJ 07087, USA. Email: msohrab1@stevens.edu 
 }}
\date{}

\usepackage{amsmath}
\usepackage{amssymb}
\usepackage{amsthm, amscd}
\usepackage{amsmath,amsthm,amsfonts,amssymb}
\usepackage[usenames]{color}

\newtheorem{thm}{Theorem}[section]
\newtheorem*{thm*}{Theorem}
\newtheorem{proposition}[thm]{Proposition}
\newtheorem{lemma}[thm]{Lemma}
\newtheorem{corollary}[thm]{Corollary}
\newtheorem{defn}[thm]{Definition}

\theoremstyle{remark}
\newtheorem{rem}[thm]{Remark}

\newcommand{\la}{\mathcal}
\newcommand{\mbb}{\mathbb}

\newcommand{\N}{\mathbb{N}}
\newcommand{\Z}{\mbb{Z}}
\newcommand{\Q}{\mbb{Q}}
\newcommand{\R}{\mbb{R}}
\newcommand{\C}{\mbb{C}}

\newcommand{\T}{\mathcal{T}}

\newcommand{\QM}{\mathbb{Q}^\times}
\newcommand{\KM}{K^\times}
\newcommand{\FM}{F^\times}

\newcommand{\RM}{R^{\times}}

\newcommand{\UTR}{\textrm{UT}_n(R)}
\newcommand{\UTF}{\textrm{UT}_n(F)}
\newcommand{\UTL}{\textrm{UT}_n(L)}

\newcommand{\TF}{\textrm{T}_n(F)}
\newcommand{\TR}{\textrm{T}_n(R)}

\newcommand{\TRf}{\textrm{T}_n(R,\bar{f})}

\newcommand{\SLF}{\textrm{SL}_n(F)}

\newcommand{\GLF}{\textrm{GL}_n(F)}

\newcommand{\GL}{\textrm{GL}}
\newcommand{\SL}{\textrm{SL}}
\newcommand{\diag}{\textrm{diag}}

\newcommand{\Lgroups}{\mathcal{L}_{\text{groups}}}

\newcommand{\one}{\textbf{1}}
\newcommand{\define}{\stackrel{\text{def}}{=}}
\newcommand{\ds}{\displaystyle}
\numberwithin{equation}{section}

\usepackage{fancyhdr}
\begin{document}

\maketitle

\begin{center}
  \large   To Gerhard Rosenberger as a token of our friendship
\end{center}

\medskip
\begin{abstract}\noindent In this paper we describe all groups that are first-order (elementarily) equivalent to the classical matrix groups such as $GL_n(F), SL_n(F)$ and $T_n(F)$ over a field $F$ provided $n \geq 3$.\\

{\bf 2010 MSC:} 03C60, 20F16\\

{\bf Keywords:} First-order equivalence, Elementary equivalence, \\ Bi-interpretability, Special linear group, General linear group, Abelian deformation.  \end{abstract}
\section{Introduction}

In this paper, we solve the first-order classification problem for the classical matrix groups $GL_n(F), SL_n(F)$ and $T_n(F)$ over fields $F$ provided $n \geq 3$, i.e., we describe the algebraic structure of all groups $H$ such that $H \equiv GL_n(F)$ or $H \equiv SL_n(F)$, or $H \equiv T_n(F)$.  Note that this gives a far reaching generalization of the famous Malcev's theorems stating that $GL_n(F) \equiv GL_m(L)$ or $SL_n(F) \equiv SL_m(L)$ imply that $n = m$ and $F \equiv L$ for all fields $F, L$ and numbers $n, m \geq 3$ \cite{Mal}. To prove this, Malcev showed that for $G \in \{GL_n(F), SL_n(F), T_n(F)\}$ the parameter $n$ and the field $F$ can be described by finitely many sentences that hold in $G$. More precisely, in the modern terminology, he showed that $n$ and $F$ are interpreted in $G$ in some nice ``uniform" way (regularly interpreted). Here, we go further and show that for such groups $G$ not only $n$ and $F$, but also the algebraic scheme itself $GL_n, SL_n$ or $T_n$, perhaps up to a cohomological abelian deformation, is regularly interpretable in $G$. It follows that any group $H$ with $H \equiv G$ is an abelian deformation of the same algebraic group as $G$ over a field $L$ with $L \equiv F$. In fact, in the case $G = SL_n(F)$ no abelian deformations are required, while for $G = GL_n(F)$ or $G = T_n(F)$ they are unavoidable. 
These groups obtained by abelian deformations from $GL_n(L)$ or $T_n(L)$ are interesting in their own right and deserve a separate study.

Our choice of $GL_n, SL_n$ and $T_n$  was mostly influenced by the desire to show various types of algebraic descriptions that may occur when one describes all models of complete theories of such groups using regular interpretations.  In \cite{MS2} the authors showed that these methods work for classical matrix groups over rings of algebraic integers. Recently,  Bunina and Gvozdevsky used this approach successfully for a very large class of Chevalley groups $G$ on rings, thus solving the first-order classification problem for such $G$ \cite{BG}. These strengthened the earlier results \cite{Segal-Tent} on  bi-interpretability with parameters for Chevalley groups due to Segal and Tent.   Avni, Lubotzky, and Meiri studied the model theory of high-rank arithmetic groups applying methods of interpretation \cite{AvniMeiri1,AvniMeiri2}. 
On the other hand, some time ago, the first-order classification problem was solved, using essentially the same technique, for algebraic nilpotent $k$-groups (unipotent $k$-groups) over arbitrary fields $k$ of characteristic zero \cite{M87}. Later, these results were extended to arbitrary finitely generated nilpotent groups \cite{MS2012}. The model theory of unitriangular groups was studied in detail in \cite{beleg94,beleg99}. All the results mentioned above indicate that there might be a general approach to the first-order classification problem of algebraic groups, based on the methods of interpretations. Furthermore, it seems that these methods and results provide an outline of the emerging unified model theory of algebraic groups, which organically includes the theory of abstract isomorphisms of such groups \cite{BM,Bunina1,Bunina2,Bunina3,Bunina4,Bunina5,BMPbook}.

\subsection{General approach}

Here, we give a general description of our approach. Throughout this paper, we assume that $n \geq 3$.

This approach is based on recent advances in the theory of interpretability (we refer to a survey~\cite{KMS}  for details), especially on the techniques of bi-interpretability ~\cite{MS1, Khelif, Nies,AKNS,Segal-Tent,KM2,KM3,MN1,DM1}.  The notion of regular interpretability is somewhat in the middle between absolute interpretability and interpretability with parameters. % (see Section \ref{se:interp})
The former one fits extremely well for questions related to the first-order equivalence, but occurs rather rarely. The interpretability with parameters is a more common one, but the parameters are usually not in the language of $G$, which presents a problem when describing models of $Th(G)$. The regular interpretability (or regular bi-interpretability) combines the useful properties of both: it can be used in the first-order classification problem in a way similar to the absolute one, and occurs almost as often as interpretability with parameters. To go further we need the following notation (see Section \ref{sec:sln}):  if a structure $\mathcal A$ is interpreted in a structure $\mathcal B$ by a set of formulas $\Gamma$ and a tuple of parameters $\bar p \in \mathcal B^n$ we write $\mathcal A \simeq \Gamma(\mathcal B,\bar p)$; if there is a formula $\phi(\bar x)$ in the language of $\mathcal B$ such that $\mathcal A \simeq \Gamma(\mathcal B,\bar p)$ for every $\bar p$ satisfying $\phi(\bar x) \in \mathcal B$, then we say that $\mathcal A$ is \emph{regularly interpretable} in $\mathcal B$ and write $\mathcal A \simeq \Gamma(\mathcal B,\phi)$.
However, showing regular interpretability (regular bi-interpretability) requires one to go deeper into model theoretic properties of the structures.   A bi-interpretability (regular bi-interpretability) of $\mathcal A$ and $\mathcal B$ is a very strong version of a mutual interpretability (regular interpretability) of $\mathcal A$ and $\mathcal B$ in each other (see Section \ref{sec:sln}).   The crucial point in using interpretability in first-order classification problems can be seen as follows: if $\mathcal A \simeq \Gamma(\mathcal B,\phi)$ and $\mathcal B_1 \equiv \mathcal B$ then for $\mathcal A_1 \simeq \Gamma(\mathcal B_1,\phi)$ one has $\mathcal A_1 \equiv \mathcal A$. Thus, regular interpretability $\mathcal A \simeq \Gamma(\mathcal B,\phi)$ gives some models $\mathcal A_1$ of the theory $Th(\mathcal A)$, but, in general, not all the models. To get all models of $Th(\mathcal A)$ in the form $\mathcal A_1 \simeq \Gamma(\mathcal B_1,\phi)$, where $\mathcal B_1 \equiv \mathcal B$ one needs regular bi-interpretability (or at least regular invertible interpretability) of $\mathcal A$ and $\mathcal B$ in each other. By construction, interpretability $\mathcal A \simeq \Gamma(\mathcal B,\phi)$ gives an algebraic description of the structure $\mathcal A$ through the structure $\mathcal B$. The efficacy of the description depends on the interpretation $\Gamma$ and how well one knows the models $\mathcal B_1$ of $Th(\mathcal B)$.

Let $G \in \{SL_n(F), GL_n(F), T_n(F)\}$ or just an algebraic group of $F$-points over a field $F$.  With $G$ we associate an algebraic structure $C(G)$, \emph{the core of $G$} which satisfy the following conditions:
\begin{enumerate}
\item [1)] $C(G)$ determines the essential algebraic properties of $G$;
\item [2)] $C(G)$ is regularly interpretable in $G$, so $H \equiv G$ implies $C(H) \equiv C(G)$;
\item [3)] $C(G)$ is bi-interpretable (or invertible interpretable) with $F$, i.e.,  $C(G) \simeq \Gamma(F,\phi)$. Hence 
$$
C(H) \equiv C(G) \Longleftrightarrow C(H) \simeq \Gamma(L,\phi) \ for \ some \ L \equiv F.
$$

\end{enumerate}

Now if $H \equiv G$ then by 2) $C(H) \equiv C(G)$,  hence by 3) $C(H)$ has the same structure over some field $L$, as $C(G)$ over the field $F$, where $L \equiv F$. This gives the algebraic structure of the core $C(H)$ over $L$. Knowing the core $C(H)$ one deduces the algebraic structure of $H$. This description of $H$ depends on how well the core $C(G)$ ($C(H)$) reflects the structure of $G$ ($H$). 
Our main results below explain how this approach works in particular cases when $G \in \{SL_n(F), GL_n(F), T_n(F)\}$.

The structure of the paper is as follows.  In Section~\ref{prelim:sec} we collect all the basic notation and results on matrix groups.  In Section~\ref{se:GLn}, we recall some results from \cite{MS3} which show that many important subgroups of the classical matrix groups $GL_n(R)$ and $PGL_n(R)$ are Diophantine, that is, definable by Diophantine formulas. These are the basis for our bi-interpretability results. In passing, we review the principal results on the decidability of Diophantine problems in the classical matrix groups $G$, which generally state that Diophantine problems in $G$ and $F$ are P-time equivalent. Notice that recently similar results were obtained for Chevalley groups over arbitrary rings \cite{BMP}, namely, if $\Phi$ is an irreducible root system of rank $> 1$, $R$ is an arbitrary commutative ring with~$1$, then  the Diophantine problem in any Chevalley group
$G_\pi(\Phi,R)$ is Karp equivalent to
the Diophantine problem in~$R$. Surprisingly, this type of reduction of Diophantine problems in groups to Diophantine problems in rings occurs not only in algebraic groups, where it is of course very natural, but also in various classes of discrete groups. Thus,  in \cite{GMO} it was shown that the Diophantine problem in wide classes of finitely generated solvable groups $G$ is closely related to the Diophantine problem in the corresponding rings of algebraic integers. Furthermore, a similar approach works for various non-commutative $R$-algebras $A$, in which case the Diophantine problem in $A$ is equivalent to the Diophantine problem in a ring $R$ or in a ring of polynomials $R[X]$, or in some other commutative ring associated with $A$ (see, for example, \cite{M87,GMO_rings,GMO_comm,KM2,KM3}). 

The rest of the paper is devoted to proving the main results of this paper, which we describe without technical details in the following three subsections.

\subsection{Groups first-order equivalent to $SL_n(F)$}

Let $G = SL_n(F)$. In this case, the core of  $SL_n(F)$ is  $SL_n(F)$ itself, so conditions 1) and 2) above on the core $C(G)$ obviously hold, while condition 3) follows from the result below. 

\medskip
\noindent
{\bf Theorem A1}
 {\it Let  $F$ be a field and $n \geq 3$. Then the group $SL_n(F)$ is regularly bi-interpretable with $F$. }
 
  \medskip
Based on this, we describe all groups first-order equivalent to $SL_n(F)$.

\medskip\noindent
{\bf Theorem A2} {\it 
 Let $F$ be a field and $n\geq 3$. Then for a group $H$ the following conditions are equivalent:
 \begin{itemize}
     \item [1)] $H \equiv SL_n(F)$;
     \item [2)] $H \simeq SL_n(L)$ for some field $L$ such that $F \equiv L$.
      \end{itemize}
}

\subsection{Groups first-order equivalent to $T_n(F)$}

Let $G =  T_n(F)$.  In this case, we define the core $C(T_n(F))$ as the quotient $P\TF=\TF/Z(\TF)$. Clearly,  $P\TF$ is absolutely interpretable in $T_n(F)$. 

The following is condition 3) on the core $C(T_n(F))$. 

\medskip
{\bf Lemma \ref{tnmainbiinter2:lem}.} {\it The quotient group $P\TF=\TF/Z(\TF)$ is regularly bi-interpretable with $F$.}

\medskip 
As a corollary we get the following crucial technical result.

\medskip
{\bf Corollary  \ref{TRF:cor}.} {\it  If $H$ is a group and $H\equiv \TF$, then 
$$H/Z(H) \cong PT_n(L)$$ for some field $L\equiv F$.
}

\medskip
The following is the main structural theorem for groups of type $T_n$.

\medskip

{\bf Theorem A3.} {\it Assume that $F$ is an infinite field of characteristic $\neq 2$. If $H$ is any group $H\equiv \TF$, then $H\cong T_n(L,f,Z)$ for some field $L\equiv F$, symmetric 2-cocycle $f\in S^2(B_n(L),Z)$ and an abelian group $Z=Z(H)\equiv Z(\TF) \cong \FM$.}

\medskip
In the statement, $\FM$ denotes the multiplicative group of units in the field $F$, $B_n(L)\cong (L^\times)^{n-1}$, and $Z(H)$ is the center of the group $H$. The group $T_n(L,f,Z)$ is the so-called abelian deformation of $T_n(L)$. The maximal torus $D_n(L)$ consisting of all diagonal matrices in $T_n(L)\cong (L^\times)^n$ has been deformed into an abelian extension of $Z=Z(H)$ by $B_n(L)$, while the action of the deformed torus on the unipotent radical $UT_n(L)$ has remained ``the same".

The proof of this theorem is given in Section~\ref{TO:sec}. In that section, we provide simpler characterizations of the groups $H$ of Theorem A3 in the cases where $F$ is a real closed field or an algebraically closed field (see Theorem~\ref{thm:Real-Tn}). If $F$ is a number field or a similar type of field and there are some restrictions on the type of 2-cocycles, then we prove that the description in Theorem A3 is also sufficient (see Theorem~\ref{thm:sufficiency-Tn}). Finally, we prove the existence of a group $H$ elementarily equivalent to $T_n(\Q)$, where $\Q$ is the field of rationals, but the group $H$ is not of type $T_n$ (see Theorem~\ref{elemnotiso-Q:thm}).

\subsection{Groups first-order equivalent to $GL_n(F)$}

Let $G = GL_n(F)$. In this case, we put $C(G) = SL_n(F) = G^\prime$.  
Note that there is a natural number $w=w(n)$ (the width of the commutator of $G^\prime$) that depends only on $n$ and such that every element of $G^\prime= SL_n(F)$ is a product of at most $w$ transvections  (see Sections  \ref{subsec:2.1} and \ref{subsec:2.2}), so it is a product of $w$ commutators in $G$. It follows that $SL_n(F)$ is absolutely and uniformly (i.e., the defining formula does not depend on $F$) definable in $G$. This addresses conditions 2) and 3) of the definition of the core.

Observe that 
$$G= \GLF\cong \SLF \rtimes d_1(F^\times),$$
where $F^\times$ is the multiplicative group of $F$, and $d_1(F^\times) = \{diag(\alpha,1,\ldots,1) \mid \alpha \in F^\times\}$. 
 Moreover, 
 $$G/ (G'\cdot Z(G)) \cong F^\times / (F^\times)^n.$$ 
 
 This allows us to describe the algebraic structure of groups $H$ such that $H\equiv \GLF$.  We are not going to provide the main technical statements here in the introduction. However, a simpler, yet crucial, result is as follows.

{\bf Theorem A4.} {\it 
   If $H$ is any group such that $H\equiv \GLF$, then $H$ fits into the short exact sequence:
$$1\rightarrow H' \cdot Z(H) \rightarrow H \rightarrow \frac{L^\times}{(L^\times)^n} \rightarrow 1$$
where $H' \cong \SL_n(L)$, $L\equiv F$, as fields and $Z(H) \equiv F^\times$.}  

The proof of Theorem A4 and the rest of the main results on $GL_n$ are given in Section~\ref{GLO:sec}. Similarly to the case of $T_n$,  there are much simpler descriptions of the group $H$ from Theorem A4 when $F$ is an algebraically closed field or a real closed field (see Theorems~\ref{thm:comp-gln} and \ref{thm:realclosed-gln}). The most technical case belongs to number fields and similar fields, called NFT-fields (see Theorem~\ref{thm:NFT-gln}). All these more precise results are improvements of Theorem A4 for particular fields.  Finally, in Section~\ref{sec:elemnotiso-gln} we prove the existence of a group that is first-order equivalent to $GL_n(\Q)$ but not of the type $GL_n$.

\section{Preliminaries on classical matrix groups }
\label{prelim:sec}

In this section, we fix some notation and recall technical results that are used throughout the paper.

For a group $G$ and $x, y \in G$ we denote by $x^y$ the conjugate $y^{-1}xy$ of $x$ by $y$, and by $[x,y]$ the commutator $x^{-1}y^{-1}xy$. For a subset $A \subseteq G$ by $C_G(A)$ we denote the centralizer $\{x \in G \mid \forall a \in A ~([x,a] = 1) \}$, in particular,  $Z(G)=\{x\in G\mid \forall y\in G~ ([x,y] = 1) \}$ is the center of $G$. For the subsets $X, Y \subseteq G$ by $[X,Y]$, we denote the subgroup of $G$ generated by all commutators $[x,y]$, where $x \in X, y \in Y$. Then $[G,G]$ is the derived subgroup $G'$ of $G$ (the commutant of $G$).  The lower central series of $G$ is defined as $G = \gamma_1(G) \geq \gamma_2(G) \geq \ldots$, where $\gamma_{i+1}(G) = [G,\gamma_i(G)]$.

 In the rest of the paper by $R$ we denote an arbitrary associative ring with identity 1.    By $\RM$ we denote the multiplicative group of invertible (unit) elements of $R$ and by $R^+$  the additive group of $R$.

Now we define the groups we study in this paper and list some of their properties that we use often.

\subsection{$GL_n(R)$}
\label{subsec:2.1}

Fix $n \in \mbb{N}$. By $GL_n(R)$ we denote the group of all invertible $n \times n$ matrices over an associative unitary ring $R$.

Let $e_{ij}$ be an $n \times n$ matrix where the $(i,j)$ entry is $1$ and all other entries are $0$. For $1 \leq i \neq j \leq n$, the matrix $t_{ij}(\alpha)=I_n+\alpha e_{ij}$, where $\alpha\in R$ and $I_n$ is the $n\times n$ identity matrix, is called a \emph{transvection}. Sometimes we denote transvection $t_{ij}(1)$ simply by $t_{ij}$. Put $\mathcal{T}_n = \{t_{ij} \mid 1 \leq i \neq j \leq n\}$. 

Transvections $t_{ij}(\alpha), t_{kl}(\beta)$, for $\alpha, \beta \in R$, satisfy the following well-known (Steinberg) relations:
 \begin{enumerate}
 \item [1)] $t_{ij}(\alpha)t_{ij}(\beta)=t_{ij}(\alpha+\beta).$
 \item [2)] $[t_{ik}(\alpha),t_{kl}(\beta)] = t_{il}(\alpha\beta)$,  for $i\neq l$.
 \item [3)] $[t_{ik}(\alpha),t_{jl}(\beta)] =1$ for $i \neq l, j \neq k$.
  \end{enumerate} 
Observe that 2) implies that $[t_{ij}(\alpha),t_{ki}(\beta)] = t_{kj}(-\alpha\beta)$,  for $j\neq k$.

Let $diag(\alpha_1, \ldots ,\alpha_n)$ be the $n\times n$ diagonal matrix  with $(i,i)$-entry $\alpha_i\in \RM$. Then $diag(\alpha_1, \ldots ,\alpha_n) \in GL_n(R)$ and the set of all such matrices forms a subgroup $D_n(R)$ of $GL_n(R)$. Note that for any $\alpha_1, \ldots, \alpha_n \in \RM$, $\beta \in R$ the following holds:
\begin{equation} \label{eq:d-t}
 diag(\alpha_1, \ldots,\alpha_n)^{-1}t_{ij}(\beta)diag(\alpha_1, \ldots,\alpha_n) = t_{ij}(\alpha_i^{-1} \beta \alpha_j).
 \end{equation}
 In particular,
 \begin{equation} \label{eq:d-t-2}
 [t_{ij}(\beta),diag(\alpha_1, \ldots,\alpha_n)] = t_{ij}(\alpha_i^{-1} \beta \alpha_j -\beta).
 \end{equation}

 By $d(\alpha)$ we denote the scalar matrix $diag(\alpha, \ldots ,\alpha) = \alpha I_n$, where $\alpha \in R^\times$. The set of all scalar matrices forms a subgroup $R^\times I_n \leq D_n(R)$ which is isomorphic to $R^\times$.  It follows from (\ref{eq:d-t-2}) that the subgroup $Z_n(R)$ of $R^\times I_n$, which consists of all scalar matrices $d(\alpha)$, where $\alpha$ is in the center of the group $R^\times$, forms the center of the groups $R^\times I_n$, $D_n(R)$, as well as the group $GL_n(R)$.

 Now consider the following diagonal matrices for $\alpha \in R^\times$:
 $$d_i(\alpha)\stackrel{\text{def}}{=} diag(1,\ldots ,  \underbrace{\alpha}_{i'\text{th}}, \ldots, 1).$$
 
 It is known (see, for example \cite{KM}) that if $R$ is a field then there are  natural numbers  $r$ and $s$, which depend only on $n$,  such that  every element $g \in GL_n(R)$ can be presented as a product of the type 
 \begin{equation} \label{eq:decomp}
 g= x_1 \ldots x_r d_n(\beta)y_1 \dots y_s
 \end{equation}
where  $x_i,y_j$ are transvections and $\beta \in R^\times$.

\subsection{$SL_n(R)$ and $E_n(R)$}
\label{subsec:2.2}

Denote by $E_n(R)$ the subgroup of $GL_n(R)$ generated by all transvections, i.e., $
E_n(R) = \langle t_{ij}(\alpha) \mid \alpha \in R , 1\leq i\neq j \leq n\rangle $. 

\begin{defn}
We say that a subgroup $G \leq GL_n(R)$ is \emph{large} if $G$ contains $E_n(R)$. 
\end{defn}

Throughout the paper, we assume $n\geq 3$.

If the ring $R$ is commutative, then, as usual,  the group $SL_n(R)$ consists of all matrices of $GL_n(R)$ with determinant 1. One can define $SL_n(R)$ for arbitrary division rings (using the Dieudonne determinant), but we do not consider such groups here. In every case, when we mention a group $SL_n(R)$ we assume that $R$ is commutative. 
Clearly,  $E_n(R)$ is a subgroup of $SL_n(R)$, so $SL_n(R)$ is a large subgroup of $GL_n(R)$. If $R$ is a field or Euclidean domain, then $SL_n(R) = E_n(R)$, but in general, this is not the case. For the fields $R$ we have
$$
[GL_n(R),GL_n(R)] = SL_n(R).
$$
Now (\ref{eq:decomp}) implies that 
$$
GL_n(R)  \simeq  SL_n(R) \rtimes d_n(R^\times)  \simeq GL_n(R)' \rtimes R^\times.
$$
Note also that in this case
$$
[SL_n(R),SL_n(R)] = SL_n(R).
$$

Following \cite{CK} we say that $SL_n(R)$ has \emph{bounded elementary  generation} if there is a natural number $w$ such that every element of $SL_n(R)$ is a product of at most $w$ transvections. Order all pairs of indices $(i,j), i,j = 1, \ldots,n$ into a sequence $\sigma$ in some arbitrary but fixed way, say $\sigma = (1,1), \ldots, (n,n)$. Repeat this sequence consequently $w$ times, obtaining a new sequence $\sigma^* = \sigma,\sigma, \ldots, \sigma  = (i_1,j_1), \ldots,(i_m,j_m)$, where $m = wn^2$. Then every element $g \in SL_n(R)$ can be decomposed into a product
$$
g = t_{i_1j_1}(\alpha_1) \ldots t_{i_mj_m}(\alpha_m)
$$
for some $\alpha_1, \ldots,\alpha_m \in R$ (note that we allow here elements $\alpha_i = 0$), which is uniform in order of transvections.

For $ 1 \leq i \neq j \leq n$ denote by $T_{ij}$ the one-parametric subgroup $\{t_{ij}(\alpha) \mid \alpha \in R\}$. Then the bounded elementary generation of $SL_n(R)$ is equivalent to the statement that there is sequence of pairs $(i_1,j_1), \ldots, (i_m,j_m)$, where $1\leq i_k \neq j_k \leq n$ such that 
$$
SL_n(R) = T_{i_1j_1} \ldots T_{i_mj_m}.
$$
If $R$ is a field, then formula (\ref{eq:decomp}) implies that $SL_n(R)$ has bounded elementary generation. A much more difficult argument shows that $SL_n(\mathcal{O})$ over a ring of algebraic integers $\mathcal{O}$ has bounded elementary generation \cite{CK}. However, this is not the case for even arbitrary domains. In fact, it was shown in \cite{Kallen, DV} that if $F$ is a field of infinite transcendence  degree over its
prime subfield (for example: $F= \mathbb{C}$) then for every number $c$  there is a matrix in the group $SL_n(F[x])$ that
cannot be written as a product of $c$ commutators.

\subsection{Unitriangular groups $UT_n(R)$}
\label{se:Unitriangular}

  Transvections $t_{ij}(\alpha)$, where $1\leq i<j\leq n$, $\alpha\in R$, generate the subgroup $UT_n(R)$ of all (upper) unitriangular matrices in $GL_n(R)$.

 For $m= 1, \ldots, n$ denote by $UT^m(n,R)$ the subgroup of $UT_n(R)$ consisting of all matrices with $m-1$ zero diagonals above the main one.
 Then
 \begin{equation} \label{eq:UTlowercentral}
 UT_n(R) = UT_n^1(R) > UT_n^2(R) > \ldots > UT_n^n(R) = 1.
 \end{equation}
 Furthermore, for any positive $r, s \in \mathbb{N}$ one has
 $$
 [UT_n^r(R),UT_n^s(R)] = UT_n^{r+s}(R).
 $$
 This implies that the series (\ref{eq:UTlowercentral}) is the lower central series of $UT_n(R)$, in particular, $\gamma_k(UT_n(R)) = UT_n^k(R)$.
 Direct computations show that any element $g \in UT_n^m(R)$ can be uniquely written  as a product of the following type:
 $$
 g = t_{n-m,n}(\alpha_{n-m}) t_{n-m-1,n-1}(\alpha_{n-m-1}) \ldots t_{1,1+m}(\alpha_1)h,
 $$
 where $\alpha_i \in R$ and $h \in UT_n^{m+1}(R)$.
 Therefore, 
 
 \begin{equation}\label{eq:UT-m-n}
 UT_n^m(R) = T_{n-m,n} T_{n-m-1,n-1} \ldots T_{1,1+m}UT_n^{m+1}(R)
 \end{equation}
 This implies that $UT_n(R)$ is a finite product of one parametric subgroups $T_{ij}$.

 \subsection{Triangular groups $T_n(R)$ }

 Recall that $T_n(R)$ consists of all upper triangular matrices $x = (x_{ij})$ over $R$ with units on the main diagonal, that is, $x_{ij} = 0$ for $i>j$, $x_{ij} \in R$ for $i <j$, and $x_{ii} \in R^\times$ for $1\leq i \leq n$. Clearly, $UT_n(R) \leq T_n(R)$.

 Note that any matrix $x \in T_n(R)$ can be represented as a product $x= d_x u_x$, where $d_x = diag(x_{11}, \ldots,x_{nn})$, and $u_x \in UT_n(R)$, where $u_x = (y_{ij})$, with $y_{ij} = x_{ii}^{-1}x_{ij}$ for $i<j$. 
 Therefore, $T_n(R) = D_n(R) UT_n(R)$ and $D_n(R) \cap UT_n(R) = 1$. Furthermore, $UT_n(R)$ is a normal subgroup in $T_n(R)$, see (\ref{eq:d-t}), hence
 $$
 T_n(R) \simeq UT_n(R) \rtimes D_n(R).
 $$
 
 \begin{defn}
 We call $G \leq T_n(R)$  a large subgroup of $T_n(R)$ if $G$ contains $UT_n(R)$. 
 \end{defn}
 
Obviously, $UT_n(R)$ is a large subgroup of $T_n(R)$, but it is not a large subgroup of $GL_n(R)$.

\subsection{Groups $PGL_n(R)$ and $PSL_n(R)$}

The general and special projective linear groups $PGL_n(R)$ and $PSL_n(R)$ are defined as quotients $PGL_n(R) = GL_n(R)/Z(GL_n(R))$ and $PSL_n(R) = SL_n(R)/Z(SL_n(R))$ of the groups by their centers.
Note that $Z(GL_n(R))$ is the subgroup $Z_n(R)$ of all scalar matrices $d(\alpha)$, where $\alpha$ belongs to the center of the group $R^\times$. Respectively, $Z(SL_n(R)) = Z_n(R) \cap SL_n(R)$. 

\begin{defn}
Let $\phi: GL_n(R)\to PGL_n(R)$ be the canonical homomorphism. We call a subgroup $G$ of $PGL_n(R)$ \emph{large} if it contains $\phi(E_n(R))$.
\end{defn}

\section{Diophantine structure in large subgroups of classical matrix groups}
 \label{se:GLn}
 
 In this section, we recall some results from \cite{MS3} which show that many important subgroups of the classical matrix groups $GL_n(R)$ and $PGL_n(R)$  are Diophantine. 
 We freely use the notation from Preliminaries.

\label{se:GLn_Tij}
We start with the following key technical result of \cite{MS3}.

\begin{thm} \cite{MS3} \label{le:Tkl}
Let $G$ be a large subgroup of $GL_n(R)$, $n \geq 3$. Then for any $1\leq k \neq m \leq n$  the subgroup $T_{km}$ is Diophantine in $G$ (defined with parameters $\T_n$). 
\end{thm}

\begin{corollary}\cite{MS3}
Let $G =  GL_n(R)$, $n \geq 3$. Then for any $1\leq k \neq m \leq n$  the subgroup $T_{km}$ is Diophantine in $G$ (defined with parameters $\T_n$). 
\end{corollary}
 
\begin{corollary} \cite{MS3}
Let $R$ be a commutative ring and $G =  SL_n(R)$, $n \geq 3$. Then for any $1\leq k \neq m \leq n$  the subgroup $T_{km}$ is Diophantine in $G$  (defined with parameters $\T_n$). 
\end{corollary}

\label{se:UT_n}

\begin{proposition} [\cite{MS3}] \label{le:UT(n,R)}
Let $G$ be a large subgroup of $GL_n(R)$, $n \geq 3$.
Then for every $1\leq m\leq n-1$ the subgroup $UT_n^m(R)$ is Diophantine in $G$. In particular, $UT_n(R)$ is Diophantine in $G$.
\end{proposition}

\begin{lemma} [\cite{MS3}]
Let $G$ be a large subgroup of $GL_n(R)$, $n \geq 3$. Then the set $R_G$ of all scalar matrices from $G$ is Diophantine in $G$.
\end{lemma}

\begin{proposition} [\cite{MS3}] \label{le:D}
Let $G$ be a large subgroup of $GL_n(R)$, $n \geq 3$.  If $R^+$ does not have elements of order 2, then the following hold:
\begin{itemize}
    \item [1)] $G\cap D_n(R)$ is Diophantine in $G$.
    \item [2)] $G\cap T_n(R)$ is Diophantine in $G$.
\end{itemize}
\end{proposition}

\subsection{Mutual Diophantine interpretability of the classical matrix groups and $R$}

\subsubsection{The Diophantine  interpretability of $R$ in classical linear groups}

\label{se:Dioph_6}

\begin{thm} \label{th:Rinterp}
Let $G$ be a large subgroup of $GL_n(R)$, $PGL_n(R)$ (assuming that $R$ has no zero divisors in this case), or $T_n(R)$, $n \geq 3$. Then the ring $R$ is e-interpretable in $G$.
\end{thm}
\begin{proof}
There are three cases to consider.

Case 1. $G$ is a large subgroup of $GL_n(R)$.
By Theorem  \ref{le:Tkl}  the subgroups $T_{12}$, $T_{2n}$, and $T_{1n}$ are Diophantine in $G$. 

We e-interpret $R$ on $T_{1n}$, turning it into a ring $\langle T_{1n}; \oplus, \otimes \rangle$ as follows. 

For $x, y \in T_{1n}$ define
\begin{equation} \label{eq:oplus}
    x \oplus y = x\cdot y
\end{equation}

Note that if $x = t_{1n}(\alpha), y = t_{1n}(\beta)$, then $x\cdot y = t_{1n}(\alpha + \beta)$, which corresponds to the addition in $R$. 

To define $x\otimes y$ for given $x, y \in T_{1n}$ we need some notation. Let $x_1, y_1 \in G$ be such that  
\begin{equation} \label{eq:e-interpR1}
x_1 \in T_{12} \text{ and } [x_1,t_{2n}] = x, \ \ y_1 \in T_{2n} \text{ and } [t_{12},y_1] = y.
\end{equation}
Note that such $x_1, y_1$ always exist and unique, for if $x = t_{1n}(\alpha), y = t_{1n}(\beta)$ then $x_1 = t_{12}(\alpha), y_1 = t_{2n}(\beta)$.
Now, define
$$
x\otimes y = [x_1,y_1].
$$
Observe, that in this case 
\begin{equation} \label{eq:e-interpR2}
[x_1,y_1] = [t_{12}(\alpha),t_{2n}(\beta)] = t_{1n}(\alpha\beta),
\end{equation}
so $\otimes$ corresponds to the multiplication in $R$. To finish the proof, we need two claims.

Claim 1. The map $\alpha \to t_{1n}(\alpha)$ gives rise to a ring isomorphism $R \to \langle T_{1n}; \oplus, \otimes \rangle$. 

This is clear from the argument above.

Claim 2. The ring $\langle T_{1n}; \oplus, \otimes \rangle$ is e-interpretable in $G$. 

To see this, observe first that, as mentioned above, $T_{1n}$ is Diophantine in $G$. The addition $\oplus$ defined in (\ref{eq:oplus}) is clearly Diophantine in $G$. Since the subgroups $T_{12}$ and $T_{2n}$ are Diophantine in $G$ the conditions (\ref{eq:e-interpR1})  are Diophantine in $G$, as well as the condition (\ref{eq:e-interpR2}).  This shows that the multiplication $\otimes$ is also Diophantine in $G$. This proves case 1.

Case 2. Let $G$ be a large subgroup of $PGL_n(R)$. The proof is similar to the one in Case 1. Only instead of
Theorem~\ref{le:Tkl}, one uses a similar result for $PGL_n$ (see \cite{MS3}).

Case 3. Let $G$ be a large subgroup of $T_n(R)$. To prove that $R$ is e-interpretable in $G$ we adjust the above argument by making a few changes. That is, we replace the subgroup $T_{12}$ by the Diophantine subgroup $T_{12}' = R_GT_{12}T_{1n}$ if $n>3$. If $n=3$, then we also replace $T_{2n}=T_{23}$ with $T_{23}'=R_GT_{23}T_{13}$. In both cases, the argument works word by word except that in these cases the elements $x_1 \in T_{12}'$ and $y_1\in T_{23}'$ if $n=3$, are not unique. However, this does not matter, since any such $x_1$ and $y_1$ give the same commutator $[x_1,y_1]$.  

 This proves the theorem.
\end{proof}

\begin{corollary} \label{co:inter1}
The ring $R$ is e-interpretable in the groups $GL_n(R)$, $SL_n(R)$ (assuming that in this case $R$ is commutative), $E_n(R)$, $T_n(R)$ and $UT_n(R)$, where $n \geq 3$. If in addition $R$ has no zero divisors, then $R$ is e-interpretable in $PGL_n(R)$ and $PSL_n(R)$ (as before, $R$ is also commutative in this case). 
\end{corollary}

\begin{rem} \label{re:int}
If $n\geq 3$, the ring $R$ is e-interpretable in the groups $GL_n(R)$, $PGL_n(R)$, $E_n(R)$ on each of the one-parameter subgroups $T_{ij}$. The same holds for $T_n(R)$ and $UT_n(R)$, if $j-i\geq 2$. If in addition $R$ is commutative, then it can be e-interpretable in $SL_n(R)$ and $PSL_n(R)$ on each of the one-parameter subgroups $T_{ij}$. 
\end{rem}

\subsubsection{The Diophantine absolute interpretability of classical linear groups in $R$}

Now we prove the converse of Theorem \ref{th:Rinterp} (with the exception of $E_n(R)$). We believe that the result is known in folklore.

\begin{proposition} \label{pr:inter2}
 The groups $GL_n(R)$, $T_n(R)$, and $UT_n(R)$,  are all e-interpretable in $R$. If $R$ is commutative, then the groups $PGL_n(R)$, $SL_n(R)$, and $PSL_n(R)$   are all e-interpretable in $R$.
\end{proposition}
\begin{proof}
We represent an $n\times n$ matrix $x = (x_{ij})$ with entries in $R$ by an $n^2$-tuple $\bar x$ over $R$, where
$$
\bar x = (x_{11}, \ldots,x_{1n}, x_{21}, \ldots, x_{n1}, \ldots, x_{nn}).
$$
The matrix multiplication $\odot$ on tuples from $R^{n^2}$ is defined by
$$
\bar x \odot \bar y = \bar z \Longleftrightarrow \bigwedge_{i,j = 1}^n z_{ij} = P_{ij}(\bar x,\bar y),
$$
where $P_{ij}(\bar x, \bar y)$ is an integer polynomial $\Sigma_{s= 1}^n x_{is}y_{sj}$.
The multiplication $\odot$ is clearly Diophantine. 
To finish the description of the interpretations of the groups $GL_n(R)$, $SL_n(R)$, $T_n(R)$ and $UT_n(R)$ in $R$, it suffices to define the corresponding subsets of $R^{n^2}$ using Diophantine formulas. We do it case by case.

The case of $GL_n(R)$.   A matrix $x = (x_{ij})$ belongs to $GL_n(R)$ if it is invertible, that is, there exists a matrix $y = (y_{ij})$ such that $xy = I_n$. In the language of the tuples $\bar x$ and $\bar y$, this is expressed as
$$
\exists \bar y (\bar x \odot \bar y) = \bar I_n.
$$
This condition is Diophantine, so $GL_n(R)$ is e-interpretable in $R$.

The case of $SL_n(R)$. In this case $R$ is a commutative ring. Recall that the determinant of a $n\times n$ matrix $(x_{ij})$ with entries in a commutative ring $R$ can be calculated as the value of some fixed polynomial $Det_n(\bar x)$ in the tuple $\bar x$. Therefore, the matrix $x= (x_{ij})$ belongs to $SL_n(R)$ if and only if $Det_n(\bar x) = 1$. This is a Diophantine equation in $R$. 

The case of $PGL_n(R)$ and $PSL_n(R)$.   Observe that if the ring $R$ is commutative, then the subgroup of all scalar matrices in $GL_n(R)$ and $SL_n(R)$ is Diophantine as the centralizer of the finite set of all transvections $t_{ij}$. Therefore, the factor groups $PGL_n(R)$ and $PSL_n(R)$ are e-interpretable, correspondingly,  in the groups $GL_n(R)$ and $SL_n(R)$. Now, the result follows from the previous two cases by the transitivity of e-interpretability.

The case of $T_n(R)$. By definition $x= (x_{ij}) \in T_n(R)$ if and only if $x_{ij} =0$ for all $i<j$ and there exists $y$ in $R$ such that $x_{11} \ldots x_{nn} y = 1$. All these conditions are Diophantine, so $T_n(R)$ is e-interpretable in $R$.

The case of $UT_n(R)$ is similar to the one above.
This proves the proposition.

\end{proof}

\begin{thm} \label{th:undec}
    Let $n \geq 3$. The Diophantine problem in each of the groups $GL_n(R)$, $SL_n(R)$ (assuming that in this case $R$ is commutative), $T_n(R)$, and $UT_n(R)$, as well as in the groups $PGL_n(R)$  and $PSL_n(R)$ (in this case we additionally assume that the ring $R$ has no zero divisors) is Karp equivalent to the Diophantine problem in $R$. 
    In particular, the Diophantine problem in all these groups is decidable if and only if it is decidable in $R$.
\end{thm}
\begin{proof}
The result follows from Corollary \ref{co:inter1}, Proposition \ref{pr:inter2}, and the properties of e-interpretability (see the complete details in~\cite{MS3}). 
\end{proof}

We cannot show that $E_n(R)$ is e-interpretable in $R$ for any associative ring $R$. However, the following holds.  

\begin{thm}
If $E_n(R)$ has bounded elementary generation then $E_n(R)$ is e-interpretable in $R$. In this case, the Diophantine problem in $E_n(R)$ is Karp equivalent to the Diophantine problem in $R$. 
\end{thm}

\begin{corollary}
If $n \geq 3$, and $E_n(R)$ has bounded elementary generation, then the Diophantine problem in $E_n(R)$ is Karp equivalent to the Diophantine problem in $R$.
\end{corollary}

For an arbitrary associative ring $R$, we have the following consequence of Corollary \ref{co:inter1}. 

\begin{proposition}
If the Diophantine problem is undecidable in a ring $R$, then it is also undecidable in the group $E_n(R)$ for any $n \geq 3$.
\end{proposition}
\section{Groups first-order equivalent to $SL_n(F)$}\label{sec:sln}
\subsection{Bi-interpretability of $R$ and $SL_n(R)$}

Let $\T_n$ be a tuple of all transvections $t_{ij}(1), 1\leq i \neq j \leq n$ ordered in some arbitrary but fixed way.

\begin{lemma} \label{le:int of R in G}
 Let $R$ be an associative commutative unitary ring. Then $R$ is interpretable in every large subgroup of $SL_n(R)$, $n \geq 3$ using Diophantine formulas with constants from $\T_n$.
\end{lemma}
\begin{proof}

 Let $G$ be a large subgroup of $SL_n(R)$.
By Theorem \ref{le:Tkl}, the subgroups $T_{12}$, $T_{2n}$, and $T_{1n}$ are Diophantine in $G$ with parameters $\T_n$. 

Similarly to the proof of Case 1 in Theorem~\ref{th:Rinterp} we interpret $R$ on $T_{1n}$ turning it into a ring $\langle T_{1n}; \oplus, \otimes \rangle$.
This proves the lemma, but we would like to note that similar to the ring $\langle T_{1n}; \oplus, \otimes \rangle$ one can interpret the ring $R$ in $G$ as $\langle T_{ik}; \oplus, \otimes \rangle$ for any  $1\leq i \neq k \leq n$. In fact, in this case one may use the Diophantine subgroups $T_{ij}$, $T_{jk}$ and $T_{ik}$, instead of $T_{12}, T_{2n}$ and $T_{1n}$ where $j$ is an arbitrary index not equal to $i$ and $k$.  For further references  we will use notation
\begin{equation} \label{int-ik}
R_{ik} = \langle T_{ik}; \oplus, \otimes \rangle.
\end{equation}

\end{proof}

\begin{lemma} \label{le:connected iso}
Let $R$ be an associative commutative unitary ring, $n \geq 3$, and $1\leq i\neq j \leq n$, $1\leq k \neq m\leq n$. Then the following hold:
\begin{itemize}
    \item [1)] The map 
$f_{i,j,k,m}: T_{ij} \to T_{km}$ defined by
$$
f_{i,j,k,m}: t_{ij}(\alpha) \to t_{km}(\alpha), \ \alpha \in R, 
$$
is Diophantine  in $G$ with parameters $\T_n$;
\item [2)] The map $f_{i,j,k,m}$ defined above gives an isomorphism of interpretations $R_{ij} \to R_{km}$ of the ring $R$ in $G$.
\end{itemize}
\end{lemma}
\begin{proof}
Fix arbitrary $1\leq i\neq j \leq n$ and $1\leq k \neq m\leq n$.

Case 1. Let $i = k$. Note that in this case $j \neq m$. Take $x = t_{ij}(\alpha) \in T_{ij}$. Then $f_{ijkm}(x) = y$ where $y = t_{km}(\alpha) =  [t_{ij}(\alpha),t_{jm}(1)] = [x,t_{jm}(1)]$.

Case 2. Suppose $j = m$. Note that in this case $i \neq k$. Take $x = t_{ij}(\alpha) \in T_{ij}$. Then $f_{ijkm}(x) = y$ where $y = t_{km}(\alpha) = [t_{ki}(1),t_{im}(\alpha)] = [t_{ki}(1),x]$.

Observe that both maps $f_{ijim}$ and $f_{imkm}$ are Diophantine in $G$ with parameters $\T_n$.

Case 3. Assume $i \neq k$ and $j \neq m$. In this case, the map $f_{ijkm}$ is the  composition of the two maps
$f_{ijkm} = f_{imkm} \circ f_{ijim}$ each of which is Diophantine in $G$. Hence,  $f_{ijkm}$ is also Diophantine in $G$ as the composition of Diophantine maps.

\end{proof}

\begin{lemma} \label{le:inter of SL(n,R) in R}
 Let $R$ be an associative commutative unitary ring and $n \geq 3$. Then the group $SL_n(R)$ is absolutely interpretable by equations in $R$.
\end{lemma}
\begin{proof}
We represent an $n\times n$ matrix $x = (x_{ij})$ with entries in $R$ by an $n^2$-tuple $\bar x$ over $R$, where 
$$
\bar x = (x_{11}, \ldots,x_{1n}, x_{21}, \ldots, x_{n1}, \ldots, x_{nn}).
$$
The matrix multiplication $\odot$ on tuples from $R^{n^2}$ is defined by 
$$
\bar x \odot \bar y = \bar z \Longleftrightarrow \bigwedge_{i,j = 1}^n z_{ij} = P_{ij}(\bar x,\bar y),
$$
where $P_{ij}(\bar x, \bar y)$ is   integer polynomial $\Sigma_{s= 1}^n x_{is}y_{sj}$.
Thus, the multiplication $\odot$ is interpretable by the equations in $R$. 
To finish the description of the interpretation, it suffices to define the subset of $R^{n^2}$ which corresponds to $SL_n(R)$ by equations. We do it case by case.

 Recall that the determinant of an $n\times n$-matrix $(x_{ij})$ with entries in a commutative ring $R$ can be calculated as the value of some fixed polynomial $Det_n(\bar x)$ in the tuple $\bar x$. Therefore, the matrix $x= (x_{ij})$ belongs to $SL_n(R)$ if and only if $Det_n(\bar x) = 1$. This gives an equation that defines $SL_n(R)$ in $R$.  

\end{proof}

\begin{rem}
We refer to the interpretation of $SL_n(R)$ in $R$ from Lemma \ref{le:inter of SL(n,R) in R} as the \emph{standard interpretation}.
\end{rem}

{\bf Theorem A1.} {\it Let $F$ be a field, and $n \geq 3$. Then the group $SL_n(F)$ is regularly bi-interpretable with $F$.}
%label{th:bi-int SL and F}

\begin{proof}
 Let $G = SL_n(F)$. Any field, hence  $F$,  has the elementary bounded generation property. Therefore, there is a number $m$ such that every element $g \in G$ is a product of at most $k$ transvections. It follows that there is fixed sequence $\sigma$ of pairs of indexes $(i_1,j_1), \ldots,(i_m,j_m)$ such that  for every element $g \in G$ there exists a sequence of elements $\alpha_1, \ldots, \alpha_m \in F$ such that 
 \begin{equation} \label{eq:A}
 g = t_{i_1j_1}(\alpha_1) \ldots t_{i_mj_m}(\alpha_m).
 \end{equation}
 In this case, we will write $g =  t_\sigma(\bar \alpha)$
 
 Claim 1. There is a Diophantine formula $\Phi(x,\bar y,\T_n)$ with parameters $\T_n$ (here $y = (y_1, \ldots,y_m)$) such that for any elements $g,h_1, \ldots,h_m \in G$ the following equivalence holds: 
 $$
 G \models \Phi(g,\bar h,\T_n) \Longleftrightarrow  \bigwedge_{i = 1}^m h_i = t_{1n}(\alpha_i) \wedge  g = t_\sigma(\bar \alpha).
 $$
 Indeed, the subgroup $T_{1n}$ is defined by the equations in $G$. For each element $t_{1n}(\alpha_k) \in T_{1n}$ by Lemma \ref{le:connected iso} the definable map (by Diophantine formulas with parameters in $\T$) $f_{1ni_kj_k}: T_{1n} \to T_{i_kj_k}$ give  the transvection $t_{i_kj_k}(\alpha_k)$ as the image $f_{1ni_kj_k}(t_{1n}(\alpha_k))$.
 Now to finish the proof of the claim, it suffices to note that 
 $$
 g = t_\sigma(\bar \alpha) \Longleftrightarrow g = \Pi_{k=1}^m f_{1ni_kj_k}(t_{1n}(\alpha_k)).
 $$
Now consider the interpretation $F_{1n} = \Gamma(G,\T_n)$  of $F$ in $G$ given in Lemma \ref{le:int of R in G} and also the standard matrix interpretation $G^* = \Delta(F_{1n})$ of $SL_n(F)$ in $F_{1n}$ from Lemma \ref{le:inter of SL(n,R) in R}.

Claim 2. There is an isomorphism $\lambda: G \to G^*$ definable in $G$ by Diophantine formulas  with parameters $\T_n$.

Let $M(t_\sigma(\bar \alpha))$ be the matrix given by the product of transvections $t_\sigma(\bar \alpha) = \Pi_{k=1}^m f_{1ni_kj_k}(t_{1n}(\alpha_k) $. Since the sequence $\sigma$ is fixed then for every $(i,j)$ with $1\leq i\neq j \leq n$, there exists a
 polynomial $P_{ij}$ with integer coefficients such that for any $\bar \alpha \in F^m$  the entry $a_{ij}$ of the matrix $M(t_\sigma(\bar \alpha))$ is equal to $P_{ij}(\bar \alpha)$. Now we have an equivalence 
 $$
 g = t_\sigma(\bar \alpha) \Longleftrightarrow \lambda(g) = (a_{ij}) \bigwedge_{i,j = 1}^n a_{ij} = P_{ij}(\bar \alpha).
 $$
 which gives an isomorphism $\lambda$. The equality $g = t_\sigma(\bar \alpha)$ is described by the formula $\Phi(x,\bar y,\T_n)$ of claim 1. Hence, the following formula defines the isomorphism $\lambda$ (here by $\bar t_\sigma(\bar \alpha)$ we denote the sequence $(t_{i_1j_1}(\alpha_1), \ldots, t_{i_mj_m}(\alpha_m))$:
 $$
 \exists \bar t_\sigma(\bar \alpha) [\Phi(g,\bar t_\sigma(\bar \alpha), \T_n) \wedge \lambda(g) = (a_{ij}) \bigwedge_{i,j = 1}^n a_{ij} = P_{ij}(\bar \alpha)].
 $$
 Note that the defining formula is Diophantine with parameters $\T_n$.
 
 Consider now the standard interpretations $G^* = \Delta(F)$ and the interpretation $F^* = \Gamma(G^*,\T_n^*)$ from Lemma \ref{le:int of R in G}, where $\T_n^*$ is the set of transvections in the group $G^* \simeq SL_n(F)$.
 
 Claim 3. There is an isomorphism $\mu: F \to F^*$ definable in $F$ by Diophantine formulas  with parameters $\T_n^*$.
 
 We define the isomorphism $\mu: F \to F^*$ by $\mu(\alpha) = t_{1n}(\alpha)$ where  $t_{1n}(\alpha)$  is the corresponding matrix in the group  $G^*$. Hence the condition that defines $\mu(\alpha)$ in $F$ is the following:
 $$
\mu(\alpha) = (a_{ij}) \wedge a_{1n} = \alpha \bigwedge_{j=1}^n a_{ii} = 1 \bigwedge_{1\leq i\neq j\leq n, i \& (i,j) \neq (1,n)} a_{ij} =0.
$$
 
 The claims 1 and 2 show that $F$ and $G = SL_n(F)$ are bi-interpretable with parameters $\T_n$. Now we need to show that one can take any tuple of parameters from some definable subset of tuples in $G$ to bi-interpret $F$ and $G = SL_n(F)$.

 Now we need to show that the bi-interpretation of $G$ and $F$ described above is regular. Note first that the interpretations $G^* =  \Delta(F)$ and $F^* = \Gamma(\Delta(F),\T_n^*)$ are absolute (the second is absolute because the tuple $\T_n^*$ is absolutely definable in $F$. 
  Therefore, it suffices to show that instead of $\T_n$ one may use any tuple from a definable subset of tuples in $G$. 
  
  Consider the following conditions:  
  
  \begin{itemize}
  \item the interpretation $\Delta(F)$ gives a group;
  \item the interpretation $F \simeq \Gamma(G,\T_n)$ gives a field;
  \item the formula that defines $\lambda$ gives an isomorphism of groups;
  \item the formula that defines $\mu$ gives an isomorphism of fields. 
  \end{itemize}
  
  As we have seen above, all the conditions above can be described by a formula with parameters $\T_n$. Let $\Psi(\bar z)$ be the conjunction of formulas that describe the conditions above and where the parameters $\T_n$ are replaced by a tuple of new variables $\bar z$. Denote by $D$ the definable set $\{S \mid G \models \Psi(S)\}$. Observe that $G \models \Psi(\T_n)$, $D \neq \emptyset$.
  
  Let $S \in D$. Then $F_S = \Gamma(G,S)$ is a field, $G_S = \Delta(F_S)$ is the group isomorphic to $SL_n(F_S)$ (by construction of $\Delta(F_A)$). Furthermore, the formula that defines the isomorphism $\lambda:G \to G^*$  now defines some isomorphism $\lambda_S:G \to G_S$, so $SL_n(F) \simeq SL_n(F_S)$. By the abstract isomorphism theorem (see the book\cite{BMPbook}) one has $F \simeq F_S$. This shows that the interpretation $F \simeq \Gamma(G,D)$ of $F$ in $G$ is regular with respect to the definable set $D$. As we have mentioned above, the interpretation $G \simeq \Delta(F)$ is absolute and hence also regular. The formulas that define the isomorphisms $\lambda_S$ and $\mu_S$ are also uniform in $S \in D$. This shows that $G$ and $F$ are regularly bi-interpretable.

  This proves the theorem.
 
 \end{proof}
 
 \begin{rem}
 In the proof of Theorem A1. above one can replace the field $F$ with any associative commutative  unitary  ring $R$ which satisfy the following conditions:
 \begin{itemize}
     \item $R$ admits bounded elementary generation (in fact, this could be an even weaker condition);
     \item $R$ satisfies abstract isomorphism theorem for $SL_n$, $n \geq 3$ (there are many such rings,  see \cite{BMPbook}).
 \end{itemize}
  We call such rings {\it $SL$-capable}.
 \end{rem}
 
\subsection{The characterization theorem }

 \begin{thm}\label{th:elem-iso-sl}
 Let $F$ be a field, and $n\geq 3$. Then for a group $H$ the following conditions are equivalent:
 \begin{itemize}
     \item [1)] $H \equiv SL_n(F)$;
     \item [2)] $H \simeq SL_n(L)$ for some field $L$ such that $F \equiv L$
     
 \end{itemize}
 \end{thm}
 \begin{proof}
 We use notation from Theorem A1.
     1) $\to$ 2). Suppose that $H$ is a group such that $H \equiv SL_n(F)$.  Since $F \simeq \Gamma(G,S)$ for any tuple $S$ over $G$ such that $G \models \Psi(S)$ it follows that for any tuple $\tilde S$ in $H$ such that $H \models \Psi(\tilde S)$ the same formulas that interpret $F$ in $G$ with parameters $\T$ interpret a field $L_{\tilde S}$ in $H$ with parameters $\tilde S$, that is, $L_{\tilde S} \simeq \Gamma(H,\tilde S)$. Since this is true for any $\tilde S$ that satisfies $\Psi$ in $H$ it follows that $L_{\tilde S} \equiv F$ (see \cite{KMS}).    Since the isomorphism $\lambda_S:G \to G_S$ is definable by some formula $\theta(x,\bar y,S)$  with parameters $S$ for any $S \in D$ one finds that for any such $\tilde S$ the formula $\theta(x,\bar y,\tilde S)$ defines some isomorphism $\lambda^*_{\tilde S}: H \to H^*$, where $H^* = \Delta(\Gamma(H,\tilde S))$. By construction of $\Delta$ $H^* \simeq SL_n(\Gamma(H,\tilde S) \simeq SL_n(L_{\tilde S})$, which proves that 1) implies 2).
     
     2) $\to $ 1).  In fact, if $L \equiv F$ then $SL_n(F) \simeq SL_n(L)$ since $SL_n(F)$ and $SL_n(L)$ are interpretable, respectively, in $F$ and $L$ by the same formulas without parameters.

 \end{proof}

\section{Groups first-order equivalent to $\TF$}\label{TO:sec}
In this section $F$ is an infinite field of characteristic $\neq 2$, unless otherwise stated.

To keep the notation simple in this section by $\T$ we denote the set of all upper triangular transvections, $t_{ij}$, $i>j$. Let $\T^*$ denote $\T \cup \{d_i:i=1, \ldots, n\}$. 

\begin{lemma}\label{tnomainbiinter:prop} Let $H$ be the subgroup of $\TF$ generated by $\UTF$ and all the $d_i=d_i(-1)$, $i=1,\ldots ,n$, where $F$ is a characteristic zero field. Then, the field $F$ is regularly bi-interpretable with $H$.
    \end{lemma}

\begin{proof} By Theorem~\ref{th:Rinterp}, $F$ is e-interpretable in $H$ and by (an easy generalization) of the proof of Proposition~\ref{pr:inter2}, $H$ is e-interpretable in $F$. All the $T_{ij}$ are in Diophantine in $H$, with the possible exception, when $j=i+1$, where $T_{i,i+1}'=R_HT_{i,i+1}T_{1n}$ is so. Since $n\geq 3$, either $i\neq 1$ or $i+1\neq n$. Assume without loss of generality that $i\neq 1$. Then $T_{i,i+1}= [d_i, T'_{i,i+1}]$. Hence, $T_{i,i+1}$ is also Diophantine in $H$.

Note that there is a unique sequence $(i_m,j_m)$, $1\leq i_m < j_m \leq n$, $m=1, \ldots, n(n-1)/2$, such that for any $g$ in $H$, there is a unique sequence $x$ in $\langle d_1, \ldots, d_n \rangle$, and a unique sequence of $\alpha_k\in F$, such that
$$g=xt_{i1,j1}(\alpha_1)\cdots t_{im,jm}(\alpha_m)$$
The rest of the proof mimics the proof of Theorem A1.
\end{proof}

\begin{lemma}\label{Dn-Dio:lem}The subgroup $D_n(F)$ is Diophantine in $\TF$.\end{lemma}
\begin{proof} Clear, by noting that
$$x\in D_n(F) \Leftrightarrow \bigwedge_{i=1}^n [x,d_i]=1$$\end{proof}
\begin{lemma} \label{tnmainbiinter2:lem}The quotient group $P\TF=\TF/Z(\TF)$ is regularly bi-interpretable with $F$ with respect to the constants $\bar{\T^*}$, the images of the constants $\T^*$ in $P\TF$ under the canonical epimorphism.\end{lemma}
\begin{proof} Recall that $Z(G)$, where $G=\TF$, is a direct factor of $D_n(R)$, that is, there is a subgroup, say $B_n$ of $D_n$ such that $D_n=B_n \times Z(G)$. For example, we might set
$\ds B_n=\prod_{i=1}^{n-1} d_i(\FM).$ 
Let us fix the subgroup $B_n$ as defined above. To emphasize the dependency on the field $F$, we also refer to this subgroup as $B_n(F)$. Since $$G=\TF= \UTF \rtimes D_n =\UTF \rtimes (B_n \times Z(G))= (\UTF \rtimes B_n) \times Z(G) $$ clearly, $PT_n(F)\cong \UTF \rtimes B_n$. We prove the statement for $$K_n(F)=\UTF \rtimes B_n$$ where the constants are from $\T^*\setminus \{d_n\}$. We note that $$d_1\cdots d_{n-1}Z(G)= d_n^{-1}Z(G)=d_nZ(G)$$ Hence, the inclusion of all the constants $d_iZ(G)$ in the statement is not necessary. 

We proceed by proving a few claims.

\emph{Claim 1.} The subgroup $B_n$ and all the subgroups $T_{ij}$ are Diophantine in $K_n$. 

\emph{Proof of the claim.} For $B_n$ simply observe that $x\in B_n \Leftrightarrow \bigwedge_{i=1}^{n-1} [x,d_i]=1$. For $T_{ij}$, note that one simply does not need the constant $d_n$ in any of the formulas in the proof of Lemma~\ref{tnomainbiinter:prop}. In fact, any $n-1$ number of distinct $d_i$'s suffices. \qed

\emph{Claim 2.} Each subgroup $d_k(\FM)$, $k=1, \ldots, n-1$ is Diophantine in $K_n$.

\emph{Proof of the claim.} Note that:
$$\ds x\in d_k(\FM) \Leftrightarrow (x \in B_n \wedge \bigwedge_{i\neq k, j\neq k}[x,t_{ij}]=1)$$ In addition, we observe that $Z(K_n)=\one$. \qed

\emph{Claim 3.} For each $1\leq i<j\leq n$, the map $$f_{ij}: d_i(\FM) \to T_{ij}^\times, \quad d_i(\alpha)\mapsto t_{ij}(\alpha), \alpha \in \FM$$
when $T_{ij}$ is regarded as a ring, is Diophantine.

\emph{Proof of the claim.}
Take $x=d_i(\alpha)$ for some $\alpha \in \FM$, then $$y=f_{ij}(x)\Leftrightarrow y= xt_{ij}x^{-1}$$ \qed

We note that $K_n$ is still a group of matrices and given the claims above one can basically repeat the proof of Theorem A1. here. In the step to prove regularity, we note that abstract isomorphisms of groups of type $K_n$ will map the derived subgroup $(K_n(F))'=UT_n(F)$ onto the derived subgroup of $(K_n(F_s))'=UT_n(F_s)$ and then $F\cong F_S$ follows.
\end{proof}

\begin{corollary}\label{TRF:cor} If $H$ is a group and $H\equiv \TF$, then 
$$\frac{H}{Z(H)}\cong PT_n(L)$$ for some field $L\equiv F$.\end{corollary} 
\begin{proof}
This follows from Lemma~\ref{tnmainbiinter2:lem} and an argument similar to that of proof of Theorem~\ref{th:elem-iso-sl}.
\end{proof}

\subsection{Non-standard tori and abelian deformations of $\TR$} \label{abdef:sec}
For preliminaries on extension theory and the notation used here we refer the reader to~\cite{MS2} for further details  

Consider $\TR$ and the torus $D_n(R)$. The subgroup $D_n(R)$ is a direct product $(\RM)^n$ of $n$ copies of the multiplicative group of units $\RM$ of $R$. The center $Z(G)$ of $G$ consists of diagonal scalar matrices $Z(G)=\{\alpha \cdot \one: \alpha\in \RM\}\cong \RM$, where $\one$ is the identity matrix. Recall that $Z(G)$ is a direct factor of $D_n(R)$, i.e. there is a subgroup, say $B_n$ of $D_n$ such that $D_n=B_n \times Z(G)$. For example, one might set
$$B_n=\prod_{i=1}^{n-1} d_i(\RM)$$
In the sequel we fix the subgroup $B_n$ as defined above. To emphasize the dependency on the ring $R$, we also refer to this subgroup as $B_n(\RM)$.

First, we define a new abelian group $D_n(R,f,z)$, where $f\in S^2(B_n(R),Z)$ as the abelian extension:
$$1\rightarrow Z\rightarrow D_n(R,f,Z)\rightarrow B_n(R) \rightarrow 1$$
In our applications, typically $Z\cong Z(G)$ is elementarily equivalent to $\RM$ but not necessarily isomorphic to it, the group $G$ to be constructed below. If $Z\cong \RM$ we simply denote $D_n(R,f,Z)$ by $D_n(R,f)$.

%Finally we note (omitting the proof) that if $f_i\in B^2(\RM,\RM)$, for all $i=1,\ldots, n-1$, then $T_n(R)\cong T_n(R,\bar{f}).$ \end{rem}

We are now ready to define the abelian deformations $T_n(R,f,Z)$.

Recall that $G=\TR=\UTR\rtimes_\psi D_n$, where $\psi: D_n \to Aut(\UTR)$, describes the action of $D_n$ on $\UTR$ by conjugation and consider the monomorphism $\phi: D_n(R)/Z(G) \to Aut(\UTR)$ induced by $\psi$. Since $D_n(R)/Z(G)\cong D_n(R,f,Z)/Z$, we can assume that $\phi: D_n(R,f,Z)/Z \to Aut(\UTR)$. Clearly, elements of $D_n(R,f,Z)$ can be written as a product $g=(z,b)$, where $z\in A$ and $b$ is any representative of $gZ$ in $D_n(R,Z,f)$. Now define:
$$T_n(R,f,Z)\define \UTR \rtimes D_n(R,f,Z)$$
where the product is defined by:
$$(u_1,(z_1,b_1))(u_2,(z_2,b_2))= (u_1u_2^{\phi(b_1)},(z_1z_2f(b_1,b_2),b_1b_2))$$
It is not very hard to show that this is indeed a well-defined group product and $Z= Z(T_n(R,f,Z))$. 

\begin{rem}\label{cocycles:rem} Let $G=T_n(R,f,Z)$, where $Z=Z(G)$. Note that since $$\ds Ext((\RM)^{n-1},Z)\cong \prod_{i=1}^{n-1} Ext(\RM,Z)$$ $f$ is cohomologous to $\ds \prod_{i=1}^{n-1} f_i$ for some $f_i\in S^2(\RM,Z)$.
So, each 2-cocycle $f_i$, $i=1, \ldots , n-1$, defines a subgroup
$$\Delta_i(\RM)= d_i(\RM) \cdot Z$$ as an abelian extension of $Z$ by $\Delta_i(\RM)/Z\cong \RM$. For $\alpha,\beta\in \RM$, 
\[f(\alpha,\beta)=f_1(\alpha,\beta)\cdots f_{n-1}(\alpha,\beta)\]
and assume that $f_n\in S^2(\RM,Z)$ defines the subgroup $\Delta_n(\RM)$ as an extension of $Z$ by $\RM$. Since $d_1(\alpha) \cdots d_n(\alpha) \in Z(G)$, for $\alpha \in \RM$ by abuse of notation we write it as $\diag(\alpha)$. Hence, there exists a function $\diag: \RM \to Z$ such that 
\begin{align*}
diag(\alpha)diag(\beta)&= d_1(\alpha)\cdots d_n(\alpha)d_1(\beta)\cdots d_n(\beta)\\
&=d_1(\alpha)d_1(\beta)\cdots d_n(\alpha)d_n(\beta)\\
&=d_1(\alpha\beta)f_1(\alpha,\beta)\cdots d_{n-1}(\alpha\beta)f_{n-1}(\alpha,\beta)d_n(\alpha\beta)f_n(\alpha,\beta)\\
&= diag(\alpha\beta))f(\alpha,\beta)f_n(\alpha,\beta)
\end{align*}
Hence 
$f_n$ and $f^{-1}$ are cohomologous, or simply we write $f_n\equiv f^{-1}$, $``\equiv"$ not to be mistaken with elementary equivalence here.

\end{rem}

\subsection{Characterization Theorem} 

{\bf Theorem A3.} {\it Assume that $F$ is an infinite field of characteristic $\neq 2$. If $H$ is any group $H\equiv \TF$, then $H\cong T_n(L,f,Z)$ for some field $L\equiv F$, 2-cocycle $f\in S^2(B_n(L),Z)$ and an abelian group $Z=Z(H)\equiv Z(\TF) \cong \FM$ .}
\begin{proof}
Let $G= \TF$. The quotient group $G/Z(G)$ is absolutely interpretable in $G$ and the same formulas that interpret $G/Z(G)$ in $G$, interpret $H/Z(H)$ in $H$. Hence, by Corollary~\ref{TRF:cor}, 
$$H/Z(H) \cong PT_n(L)$$ 
for some $L\equiv F$. Clearly, $Z(H)\equiv  \FM \equiv L^\times$. 

%The subgroup $\sqrt{H'}$ is definable in $H$ by the same formulas which define $\sqrt{G'}\cong A\ltimes \UTF$ in $G$, where $A$ is a finite subgroup of $D_n$ generated by all the $d_i$. Therefore, by Corollary~\ref{TRF:cor}, $\sqrt{H'}\cong A \ltimes \UTL$, for some $L\equiv F$. 

\iffalse Indeed, there exists an $\Lgroups$-formula $\Phi(\bar{x},\bar{y})$ so that $G\models \Phi(\bar{t},\bar{d})$, and there exist tuples $\bar{s}$ and $\bar{e}$ of elements of $H$ such that $H\models \Phi(\bar{s},\bar{e})$utofi:lem and the one parameter subgroups $S_{ij}$ of $H$ generated by the $s_{ij}$ over $R$ generate a normal subgroup $H_u=H_u(\bar{s})$ of $H$ such that $H_u\cong \UTR$. 

For each $k=1,\ldots, n$ the subgroup $\Delta_k(G)=d_k(\FM)\cdot Z(G)$ is definable in $G$ as the subgroup of $D_n$ consisting of all $x\in G$:
$$xt_{ij}x^{-1}=\left\{\begin{array}{ll}
	t_{ij}  & \text{if } i\neq k, j\neq k \\
	t_{ij}(\alpha), \text{ for some } \alpha\in \FM&\text{if } i=k \text{ or } j=k
\end{array}\right.$$
Note that the above is expressible by $\Lgroups$-formulas. Now, for each $1\leq k\leq n-1$, there exists an $\Lgroups$-sentence expressing that $$
\forall \alpha\in \FM, \exists x\in \Delta_k \, (xt_{k,k+1}x^{-1}=t_{k,k+1}(\alpha))$$ One can make obvious adjustments to get a sentence for case $k=n$. Therefore, for each $k$ there exists an interpretable isomorphism $\Delta_k/Z(G) \to \FM$. One can also express in $\Lgroups$, through the interpretable isomorphisms mentioned above, that $Z(G)= \{\prod_{k=1}^n d_k(\alpha):\alpha\in \FM\}$.\fi

The fact that $D_n \cap G' =\one$ is uniformly expressible by $\Lgroups$-formulas. Now moving to $H$, the same formulas define a subgroup $E_n=E_n(\bar{e})$ of $H$, so that $E_n(\bar{e})/Z(H)\cong (L^\times)^{(n-1)}$ and $Z(H)\equiv L^\times$, $H'\cap E_n=\one$ and the action of $E_n$ on $H'\cong UT_n(L)$ is the action of the action of $D_n$ on the $t_{ij}\in \UTL$ extended to all of $\UTL$. This shows that $H\cong  \UTL \rtimes E_n$, where $E_n\cong D_n(L,f,Z)$, where $Z=Z(H)$, and some $f\in S^2(B_n(L), Z)$.

%It is clear now that the defining relations (1.)-(5.) of Subsection~\ref{abdef:sec} of an abelian deformation $\TRf$ must hold in $H$.
%The torsion subgroup $T(\Delta_i(G))$ of $\Delta_i(G)$ is finite. Let $N$ be its exponent. For all $n$, the sentences $\forall x\in \Delta_i(G)( x^n=1 \to x^N=1)$ hold in $G$, hence in $H$. So, the formula $x^N=1$ defines $T(\Delta_i(G))$ in $G$ as well as $T(\Delta_i(H))$ in $H$. Hence, $ T(\Delta_i(H))\cong T(\Delta_i(\OM))\cong T(\OM)\times T(\OM)$. In addition, the following holds in $G$, and consequently in $H$:
%$$\forall x\in \Delta_i(G),  \exists y\in T(\Delta_i(G)), \exists z\in Z(G)( x^N\in Z(G)  \to (x=yz ))).$$
%This ensures that in $H\equiv G$ each $f_i$ defining the $\Delta_i(H)$ as an extension of $Z(H)\cong \RM$ by $\Delta_i(H)/Z(H)\cong \RM $is CoT.

 \end{proof}
\subsection{Sufficiency of the characterization } \label{subsec:suff-Tn}

In this section, we prove a sufficiency statement for the class of fields where $\FM$ has a finite maximal torsion subgroup $T$. As $T$ is a product of finite cyclic groups, it is known that $\FM$ splits over $T$. 

We will need to state a few well-known definitions and results.

Let $B$ be an abelian group and $A$ a subgroup of $B$. Then $A$ is called a \emph{pure subgroup of $B$} if $\forall n\in \mbb{N}$, $nA=nB\cap A$.
\begin{lemma}\label{pure:lem}Let $A\leq B$ be abelian groups such that the quotient group $B/A$ is torsion-free. Then $A$ is a pure subgroup of $B$.\end{lemma}

An abelian group $A$ is called \emph{pure-injective} if $A$ is a direct summand in any abelian group $B$ that contains $A$ as a pure subgroup.

The following theorem expresses a connection between pure-injective groups and uncountably saturated abelian groups.
\begin{thm}[\cite{eklof}, Theorem 1.11]\label{ekthm} Let $\kappa$ be any uncountable cardinal. Then any $\kappa$-saturated abelian group is pure-injective.\end{thm} 
\begin{defn}[NFT-field]
  A field $F$ is called a field of \emph{number field type}, if the torsion subgroup $T$ of $\FM$ is finite. We refer to these fields as NFT-fields.
\end{defn}
 Note that if $F$ is an NFT-field then $\FM$ splits over the torsion subgroup $T$.  
\begin{defn}\label{CoT:defn}
 Let $F$ be an NFT-field with $\FM=T(\FM)\times B$ and let $Z$ be an abelian group with a similar decomposition $Z=T(Z)\times B(Z)$ where the maximal torsion subgroup is denoted by $T(Z)$ and the complement $B(Z)$ is torsion-free. We know that $Ext(\FM,Z)\cong Ext(T(\FM),A)\times Ext(B(\FM),Z)$, so any symmetric 2-cocycle $f\in S^2(\FM,Z)$ can be written as $f\equiv f_1\cdot f_2$, where $f_1\in S^2(T(\FM),Z)$, $f_2\in S^2(B(\FM),Z)$. The symmetric 2-cocycle $f$ is said to be a coboundary on torsion or CoT if $f_1$ is a 2-coboundary.\end{defn}
\begin{rem}\label{split:rem}Assume for abelian groups $Z_i$, $i=1,2$ we have $Z_i= T_i \times B_i$ where $T_i$ and $B_i$ are some subgroups of the $Z_i$. Consider a symmetric 2-cocycle $f: Z_1\to Z_2$. By abuse of notation, we consider $f$ as $f: T_1 \cdot B_1 \to T_2 \times B_2$. Then $f$ is cohomologous to $(g_1g_2, h_1h_2)$ where $g_1\in S^2(T_1,T_2))$, $g_2\in S^2(T_1,B_2)$, $h_1\in S^2(B_1,T_2)$ and finally $h_2\in S^2(B_1,B_2)$. We will use this notation in the following. %We call $f$ a \emph{torsion-free 2-cocycle} if $g_1$, $g_2$ and $h_1$ are all coboundaries. 
	
\end{rem}
\begin{lemma}\label{saturated-split:lem}  Let $F$ be an NFT-field,  and $\FM=T(\FM)\times B(\FM)$ and $Z\equiv \FM$. Then $Z=T(Z) \times B(Z)$, where $f\in S^2(\FM, Z)$ is CoT and $(I,D)$ is an ultrafilter so that the ultrapowers $(\FM)^*$ and $Z^*$ over $D$ are $\aleph_1$-saturated. Then the 2-cocycle $f^* \in S^2((\FM)^*,(Z)^*)$ induced by $f$ is a 2-coboundary.\end{lemma}
\begin{proof}
 	Since $T(\FM)$ is finite, $T(Z)$ is also finite and $Z$ splits over it. Note that $(\FM)^*=(F^*)^\times$, $T((\FM)^*)=(T(\FM))^* \cong T(\FM)$. Then $(\FM)^* \cong T \times B^*$. A similar decomposition occurs in $Z^*$. The assumption that $f$ is CoT implies that the induced cocycle $f^*\in S^2((\FM)^*,Z^*)$ is CoT. Then, using the notation of Remark~\ref{split:rem}, we see that $g_1^*\in S^2(T^*(\FM),T^*(Z))$ and $g^*_2\in S^2(T^*(\FM),B^*)$ are both coboundaries. The 2-cocycle  $h_1^*\in S^2(B(\FM)^*,T(Z)^*)$ is a coboundary since $B(\FM)^*$ is torsion-free abelian and $T(Z)^*$ is finite abelian. The 2-cocycle $h_2^*\in S^2(B(\FM)^*,B(Z)^*)$ is also a coboundary since $A^*$ is pure-injective by Theorem~\ref{ekthm} and also $A^*$ is a pure subgroup in an extension represented by $h_2^*$ because of Lemma~\ref{pure:lem}. Consequently, $f^*$ is a coboundary.
\end{proof}

\begin{thm} \label{thm:sufficiency-Tn} Let $H=T_n(F,f,Z)$, where $F$ is an NFT-field and $f$ is CoT. Then $H\equiv \TF$.\end{thm}

\begin{proof}  Let $G=\TF$, and Let $B_n(F^\times)$ be the subgroup introduced at the beginning of Subsection~\ref{abdef:sec}. 

Let $(I,D)$ be an $\aleph_1$-incomplete ultrafilter, and let $M^*$ denote the ultraproduct $M^I/D$.  Then $B_n((F^*)^\times)=B_n((\FM)^*)=B_n^*(\FM)$ is $\aleph_1$-saturated. Recall that $f$ is cohomologous with $f_1\cdots f_{n-1}$, $f_i\in S^2(\FM, Z)$. The hypothesis implies that each $f_i$ is CoT. If $f^*_i\in S^2((\FM)^*,A^*)$ denotes the 2-cocycle induced by $f_i$, then for each $i=1, \ldots,n-1$, $f^*_i$ is a 2-coboundary by the Lemma~\ref{saturated-split:lem}. The fact that $$T^*_n(F,f,Z)\cong T_n(F^*,f^*,Z^*)\cong PT_n(F^*)\times Z^* $$ requires only some routine checking. 

Therefore, \[T^*_n(F,f,Z)\cong  PT_n(F^*)\times Z^* \equiv PT_n(F^*)\times (\FM)^* \cong T_n^*(F)\] 
Hence, $T_n(F,f,Z)\equiv T_n(F)$ 
\end{proof}
\subsection{Real closed fields and algebraically closed fields}
\begin{thm} \label{thm:Real-Tn} Let $F$ be a real closed field or an algebraically closed field of characteristic $\neq 2$.  Let $G=T_n(F)$ and $H$ be a group. Then the following are equivalent 
\begin{enumerate}
    \item $G\equiv H$
    \item $H\cong PT_n(L) \times Z(H)$ for a filed $L\equiv F$ and $Z(H) \equiv \FM$.
\end{enumerate}
\end{thm} 
\begin{proof}
    We shall only prove $(1.) \Rightarrow (2.)$, as the other direction is clear. 

    So assume that $F$ is an algebraically closed field of characteristic $\neq 2$. Then $Z(H) \equiv \FM \cong Z(G)$ is a radicable (divisible) abelian group and as $H'\cong PT_n(L)$ for some $L\equiv F$, $H\cong PT_n(L) \times Z(H)$ as desired.

    If $F$ is real closed, then $\FM \cong \{\pm \one\} \times B(G)$, where $B(G)$ is the multiplicative group of positive elements of $F$, and therefore definable in $\FM \equiv Z(G)$ by:
   $$\Phi(x)=\exists y (y^2=x).$$
   Now $G'\cdot \langle d_i: i=1,\ldots n \rangle$ is absolutely definable in $G$ by $\Psi(x)= \exists x (x^2 \in G')$. This subgroup is evidently regularly bi-interpretable with $F$. As the absolutely definable subgroup $\{\pm \one\}$ splits from $G'$, in $H$ the corresponding subgroup splits from $H'$. Therefore, $Z(H)=\{\pm \one\} \times B(H)$, where $B(H)\equiv B(G)$, splits from $H$.
   
\end{proof}
 
\subsection{Abelian deformations of $T_n(\Q)$ which are not isomorphic to any $T_n(F)$ for any field $F$}\label{Tn(Q):sec}
 Let $G=\TF$. Note that since $G \cong (\UTF \rtimes B_n(F)) \times \FM$, any group of the form $H=(\UTL \rtimes B_n(L)) \times A$, where $L\equiv F$ and $A\equiv \FM \equiv Z(G)$ is elementarily equivalent to $\TF$ by a combination of the Feferman-Vaught Theorem on elementary equivalence of direct products, a simple ultraproduct argument, and Keisler-Shelah's Theorem. This is one source of elementsarily equivalent groups to $\TF$, but not isomorphic to any $T_n$. In the following, we shall construct a different type of example.  
\begin{lemma}\label{Z-inter-Q*:lem}
There exists a countable field $K$, such that $\Q\equiv F$ and also $Ext(F^\times,F^\times)\neq 1$. 
\end{lemma}
\begin{proof} By (\cite{beleg99}, Proposition 2.2.16) there exists a countable ring $R$ where $R\equiv \Z$ and $Ext(R^+,R^+)\neq 0$. In fact, for such a ring $R$, $R^+=A\oplus D$ where $A\neq 0$ is a reduced abelian group and $D\neq 0$ is a countable torsion-free divisible abelian group. Then by (\cite{fox}, 54.5) $Ext(D,A)\neq 0$, which implies that $Ext(R^+,R^+)\neq 0$.

\noindent{\bf Claim:} Let $F$ be the field of fraction of the integral domain $R$ discussed above. Then $F^\times \cong \{\pm 1\} \times A \times D$ where $A$ and $D$ are both nontrivial, $A$ is reduced and $D$ is countable, torsion-free, and divisible.

\noindent{\it Proof of the claim.} We know that $\QM= \{\pm 1\}\times C$ where $C\cong\prod_{i\in \omega}C_{\infty}$, and $C_{\infty}$ is the infinite cyclic group. First, we show that $C$ is definable in $\QM$ uniformly with respect to $Th(\Q)$. By a theorem of \cite{julia} the ring $\Z$ is definable in $\Q$.  Arithmetic $\N=\langle |\N|, +,\times , 0,1\rangle $ is definable in $\Z$. In particular, the natural order on $\Z$ is definable in $\Z$. Since $\Z$ is definable in $\Q$, it is easy to see that natural order on $\Q$ is definable in $\Q$. Now $\Q^\times$ is definable in $\Q$ as the set of invertible elements. The subgroup $C$ is definable in $\QM$ as the set of all positive elements of $\Q$. The subgroup $\{\pm 1 \}$ is obviously definable in $\QM$.  

 We note that $\Q$ is interpretable in $\Z$ and the same formulas that interpret $\Q$ in $\Z$ interpret $F$ in $R$. Thus, indeed $\Q\equiv F$. The formulas that define $\Z$ in $\Q$ define a subring $R$ of $K$ in $K$ that is elementarily equivalent to $\Z$. Again, the formulas that define $C$ in $\Q$ define a subgroup $E$ of the multiplicative group $\FM$ of $F$ in $F$. So, clearly $\FM=\{\pm 1\}\times E$. Next, we will show that $E$ has a direct factor which is isomorphic as an abelian group to $R^+$. %subgroup  $E\cong A \times D$ as above where both $A$ and $D$ are non-trivial. 
 By G\"odel's work all primitive recursive functions on $\N$ are arithmetically definable, in particular the predicate $z=2^x$ is arithmetically definable. That is, there is a first-order formula $exp(z,y,x)$ of the language of arithmetic such that
\[z=2^x \Leftrightarrow \N\models exp(z,1+1,x).\]
In particular, it can be seen that the set $2^\N=\{2^n: n\in \N\}$ has a definable arithmetic structure, also denoted by $2^\N$. Since $\N$ is definable in $\Z$, so is $2^\N$. This implies that $2^\Z$ is definable in $\Q$ and also carries a ring structure interpretable in $\Q$ and isomorphic to the ring of integers $\Z$. On the other hand, $2^\Z \subset \Q^{\times}$. Formulas that define an arithmetic structure on $2^\Z$ define in $K$ a subgroup $2^R\leq E < \KM$ of $K$ with a definable ring structure isomorphic to $R$. Now $R^+\cong 2^R \cong A'\times D'$  where both $A'$ and $D'$ are non-empty and $D'$ is torsion-free divisible. This means that $C= A \times D$ where both $A$ and $D$ are nonempty, $A$ is reduced torsion-free and $D$ is countable, torsion-free and divisible. This ends the proof of the claim.

As a corollary of the claim and the fact $Ext(D,A)\neq 1$, we can conclude that $Ext(\FM,\FM)\neq 1$.        \end{proof}

\begin{thm}\label{elemnotiso-Q:thm} There exists a field $F\equiv \Q$ and a 2-cocycle $f\in S^2((\FM)^{n-1},\FM)$, such that $T_n(F,f)\equiv T_n(\Q)=G$, but $T_n(F,f)\ncong T_n(L)$, for any field $L$.    
\end{thm}
\begin{proof} Pick the field $F$ found in Lemma~\ref{Z-inter-Q*:lem}. We can clearly find a CoT 2-cocycle $f\in S^2(\FM,\FM)$ which is not a 2-coboundary. Set, say, $f_1=f$ and the let the rest of the $f_i$ be trivial and form $H=T_n(K,f)$. By Theorem~\ref{thm:sufficiency-Tn}, we have $H\equiv T_n(\Q)$. If $\phi: H \to G$, where $G=T_n(L)$ is an isomorphism of groups, then $\UTF\cong\phi(H')\cong G' \cong \UTL$. This in fact shows that $F\cong L$.  Now, passing to the abelianizations, $\phi$  induces an isomorphism of $D_n(F,f,Z(H))$ onto $D_n(L,Z(G))$, mapping $Z(H)$ onto $Z(G)$. This is impossible, otherwise $f$ would be a 2-coboundary.\end{proof}

\section{ Groups first-order equivalent to $\GLF$, $n\geq 3$}\label{GLO:sec}

 %\subsection{Abelian Deformations of $\GLR$}\label{abdef-gl:sec}

In this section $F$ is an infinite field of characteristic $\neq 2$. The case of $\GLF$ is similar to the case of $\TF$, but with more subtlety. We know that 
$$G= \GLF\cong \SLF \rtimes d_1(F^\times)$$Hence, $\SLF=G'$, while there is a non-trivial intersection between $D_n(F)$ and $\SLF$. In fact, for any $\alpha\in F^\times$
\begin{equation}\label{eqn:d_1-to-n}
    d_1(\alpha^n)=\diag(\alpha) \prod_{i=2}^{n} d_1(\alpha) d_i(\alpha^{-1})   
\end{equation} 
while $\prod_{i=2}^{n} d_1(\alpha) d_i(\alpha^{-1}) \in \SLF$ and  $\diag(\alpha)\in Z(G)$.  On the other hand, for $G=\GLF$, since $G'=\SLF$ and $\SLF$ is boundedly generated by commutators, $G'$ is absolutely and uniformly definable in $G$. By Theorem~\ref{th:elem-iso-sl} if $H\equiv G$, then $H'\cong \SL_n(L)$ for some field $L\equiv F$, $Z(H)\equiv F^\times$ and $$G/ (G'\cdot Z(G)) \cong F^\times / (F^\times)^n$$ In fact we have proven the following statement.

{\bf Theorem A4.} {\it 
   If $H$ is any group such that $H\equiv \GLF$, then $H$ fits into the short exact sequence:
$$1\rightarrow H' \cdot Z(H) \rightarrow H \rightarrow \frac{L^\times}{(L^\times)^n} \rightarrow 1$$
where $H' \cong \SL_n(L)$, $L\equiv F$, as fields and $Z(H) \equiv F^\times$.}  

%\label{prop:gl:main}

  The above enables us to characterize some important cases.
  
  Define the \emph{isolator} of $M$ in $G$ by
$$ Is_G(M)=\{g\in G | x^n\in M, \text{ for some $n\in \N$}\}$$
Clearly $Is_G(M)$ is a normal subgroup of $G$. Now consider the following sequence of subgroups of $G=\GLF$

$$G \geq Is(G'\cdot Z(G)) \geq Is(G')\cdot Z(G) \geq Is(G')\geq G'$$  
 Now let $G=\GLF$ or, in fact, a group $G\equiv \GLF$. By Theorem A4., the sequence above becomes

 \begin{equation}\label{Sequence:egn}
   G = Is(G'\cdot Z(G)) \geq Is(G')\cdot Z(G) \geq Is(G')\geq G' \geq 1  
 \end{equation}

 We shall refer to this sequence in the section, when we try to characterize groups elementarily equivalent to $\GLF$ for some classes of fields.

\subsection{Algebraically closed fields}

\begin{thm}\label{thm:comp-gln}
Let $F$  be a characteristic zero algebraically closed field. Then the following are equivalent:
\begin{enumerate}
  \item $H\equiv \GL_n(F)$,
  \item $H\cong (H' \times Z(H))/\mu_n$, $H'\cong \SL_n(L)$, $Z(H)\equiv \C^\times$, $\mu_n =ker \phi$, $\phi: H' \times Z(H) \to H, (h,z)\mapsto hz$.   
\end{enumerate}
\end{thm}
\begin{proof}
$(1.) \Rightarrow (2.)$ Going back to the sequence in Theorem A4., $F^\times = (F^\times)^n$. As both $G'$ and $Z(G)$ are absolutely and uniformly definable in $G$, the homomorphism $\phi: \SLF \times Z(G) \to G$, $(g,a)\mapsto ga$ is absolutely and uniformly definable in $G$, as is its finite kernel $\mu_n$. Therefore, the induced isomorphism, say $\tilde{\phi}$ is absolutely and uniformly interpretable in $G$. The same formulas set an isomorphism as in (2.) in $H$.

$G=G'Z(G)$, we have $H=H'Z(H)$, while $Z(H)\equiv \C^\times$ and $H'\cong \SL_n(L)$ for some $L\equiv F$.  Note that the subgroup $\mu_n$ is isomorphic to $\{(\omega, \omega^{-1}): \omega^n=1, \omega \in \C\}$ or just the group of $n$-th roots of unity.  One might call the homomorphism $\phi$ a \emph{non-standard central isogeny}.

$(2.) \Rightarrow (1.)$  By assumption and the classical Fefferman and Vaught Theorem $G'\times Z(G) \equiv H' \times Z(H)$. Now, the same formulas that interpret $(H'\times Z(H))/\mu_n$ in $H$, interpret $(G'\times Z(G))/\mu_n$ in $G$, where $G'=\SLF$. Therefore, $H \equiv \GLF$.

\end{proof}
\subsection{Real closed fields}
\begin{lemma}\label{lem:split-models}
Let $G$, be a group that fits into the sequence of Theorem A4. where 

\begin{enumerate}
    \item $Z=Z(G)$ splits over its torsion subgroup $T(Z)$, that is, there exists a subgroup $B$ of $Z$, such that $Z=B\times T$ 
    \item $A$ is a radicable (multiplicatively divisible) group.
    \end{enumerate}
    Then $$G \cong Is_{G}(G') \times B $$
    Where $G'\cong \SL_n(L)$
\end{lemma}
\begin{proof}
    Clearly, in the situation above the sequence in~\eqref{Sequence:egn} becomes:

    $$G= Is (G')\cdot B \geq Is(G')\geq \one$$

    Where, $[B,Is(G')]=\one$, both $B$ and $Is(G')$ are normal in $G$ and $B \cap Is(G')=\one$.
\end{proof}

\begin{thm} \label{thm:realclosed-gln}
Let $F$  be a real closed field, that is, $F\equiv \R$, $\R$ the field of real numbers, and let $H$ be a group. The following are equivalent:
\begin{enumerate}
  \item $H\equiv \GL_n(\R)$, \item $H\cong \left(\SL_n(L)\cdot \langle d_1\rangle \right)\times B(H)$, where $L\equiv \R$ and $B(H)\equiv (\R_+)^\times $, where $(\R_+)^\times$ is the multiplicative group of positive reals.   
\end{enumerate}

\end{thm}
\begin{proof}
(1.) $\Rightarrow$ (2.): Clearly, $G=\GL_n(\R)$ satisfies the conditions of Lemma~\ref{lem:split-models}, so $\GL_n(\R) \cong  Is_G(\SL_n(\R)) \times B(G)$ as described while $Is_G(G')=Is_G(\SL_n(\R)) =\SL_n(L)\cdot \langle d_1\rangle$. This latter subgroup is an absolutely and uniformly definable subgroup of $G$, which is regularly bi-interpretable with $F$. Therefore, $Is(H')\cong \SL_n(\R)\cdot \langle d_1\rangle$. The subgroup $B(G)$ is definable in $Z(G)$, by the equation 
$$\Phi(x)=\exists y (y^2=x)$$    
This subgroup is both torsion-free and radicable, both properties axiomatizable with a set of first-order sentences. Hence, $B(H)\equiv (\R_+)^\times$ splits as a direct summand from the group $H$, the complement being $Is_H(H')$, whose structure has been described above. 

(2.) $\Rightarrow (1.)$: If $L\equiv \R$, then $\SL_n(L)\cdot\langle d_1\rangle \equiv \SL_n(\R)\cdot\langle d_1\rangle$ by an ultrapower argument and the fact that the action of $d_1 \in \GL_n$ on $\SL_n$ in both cases is the same. The rest follows from the Feferman-Vaught Theorem.
\end{proof}
\subsection{NFT-fields}
Let $F$ be an NFT field; then the sequence in~\eqref{Sequence:egn}, $Is(G')=\SLF \rtimes  T(d_1(\FM))$ is uniformly definable in $G=\GLF$ and clearly regularly bi-interpretable with $F$. Let $T$ be the torsion subgroup of $Z(G)$ and $G_0\leq Z(G)$ its complement. We call any subgroup $G_0$ with this description an \emph{ addition of $G$} and the quotient $G_f = G/G_0$ a \emph{foundation of $G$}. Since $T$ is a direct product of cyclic groups, such a subgroup $G_0$ always exists. Hence, $G$ is a central extension of $G_0$ by $G_f$. Note that $G_f=G/G_0$ contains a copy of $\SLF$, since $G' \cap G_0 =1$. 
\begin{lemma}\label{lem:additon-auto} Let $F$ be an NFT-field. Then, for any two additions $A_1$ and $A_2$ of $G=\GLF$, there exists an automorphism $\phi$ of $\GLF$ such that $\phi(A_1)=A_2$.
\end{lemma}
\begin{proof}
We recall that if $\eta_i: \FM/T \to \FM$ are the corresponding transversals of $A_i$'s, then $\mu: \FM/T \to T$ defined by $\mu(\alpha)=  \eta_1(\alpha)\eta_2^{-1}(\alpha)$ is indeed a homomorphism. In fact, since $\FM\cong T \times \FM/T$, we can just assume that $\mu: \FM \to T$ is defined as $\mu(T)=1$ and $\mu|_{\FM/T}= \eta_1\eta_2^{-1}$. So, to prove the statement, we need to show that there exists an automorphism $\phi$ of $\GLF$ such that $\phi(\diag(\alpha))=\diag(\alpha\cdot \mu(\alpha))$. To this end, consider the subgroup $H$ of $\GLF$ generated by $\SLF$ and the subgroup $H_1=\{d_1(\alpha)\diag(\phi(\alpha))| \alpha\in \FM\}$. In fact, $H= \SLF \cdot H_1$. All we really need to show is that if $x\in \SLF \cap H_1$, then $x=1$. So, assume that $x$ is such an element. Then $\alpha \mu^n(\alpha) =det(x)=1$. Let $\alpha = \beta \cdot t$ where $\beta$ is of infinite order and $t$ is torsion. Clearly, we must have $\beta=1$, which implies that $\mu(\alpha)=\mu(\beta)=1$. Hence $t=1$ and $x=1$.

 Next, we claim that $H=\GLF$. We already know that $\SLF \leq H$. Now we need to prove that $d_1(\alpha)\in H$ for all $\alpha \in \FM$. Write again $\alpha= \beta \cdot t$ as above. Note that if $\beta=1$, then $d_1(\alpha)diag(\mu(\alpha))=d_1(t)$. So, in fact, for all $t\in T(\FM)$, $d_1(t)\in H$. Hence $d_i(t)=d_1(t) d_1(t^{-1})d_i(t) \in H$. So for all such $t$, $diag (t)\in H$. Since $im(\mu)\leq T$, all $\diag(\mu (\alpha))\in H$, and hence $d_1(\alpha)\in H$ for all $\alpha\in \FM$. 
 
Next, define $$\phi: \GLF \to \GLF,\quad (g, d_1(\alpha))\mapsto (g,d_1(\alpha)\diag(\mu(\alpha))) $$ The map $\phi$ is a homomorphism:\\
\begin{align*}
\phi( (g_1, d_1(\alpha_1))\cdot  (g_2, d_1(\alpha_2)))&= \phi ( (g_1g_2^{d_1(\alpha_1)}, d_1(\alpha_1)d_1(\alpha_2))\\
&=  (g_1g_2^{d_1(\alpha_1)}, d_1(\alpha_1\alpha_2diag(\mu(\alpha_1\alpha_2)))\\
&=  (g_1g_2^{d_1(\alpha_1 )\diag(\mu (\alpha_1)}, d_1(\alpha_1\alpha_2\diag(\mu(\alpha_1\alpha_2)))\\
&= (g_1, d_1(\alpha_1)\diag(\mu(\alpha_1)))\cdot  (g_2, d_1(\alpha_2)\diag(\mu (\alpha_2))
\end{align*}
It is also injective. Again, if $\alpha=\beta\cdot t$ with $\beta$ of infinite order and $t$ torsion and $d_1(\alpha)\diag(\phi(\alpha))=1$, we immediately conclude that $\beta=1$ and $\mu(\alpha)=1$, implying that $t=1$, hence $\alpha=1$. This implies that $\phi(g,d_1(\alpha))=1$, $g=d_1(\alpha)=1$.

 By the above, it is also surjective, and 
$$\phi(\diag(\alpha))= \diag (\alpha\cdot \mu (\alpha))$$

\end{proof}

\begin{lemma} \label{lem:foundation-elem} Let $F$ be an NFT-field, $G=\GLF$ with an addition $G_0$ and $H\equiv G$. Then there exists a subgroup $H_0\leq Z(H)$ such that $Is(H')\cdot Z(H) = Is(H')\times H_0$ and $G_f\equiv H/H_0$.
\end{lemma}
\begin{proof}
 We proceed with an ultrapower argument. Let $D$ be an ultrafilter $I$ such that $G^I/D \cong H^I/D $. We observe that $H^*=H^I/D$ is isomorphic to $\GL_n(F^I/D)$. Set $L=F^I/D$ and assume $H^*=\GL_n(L)$. Since $G'$ is of finite width, it is not very hard to prove $(G')^*=(G^*)'$ and then to conclude that $Is((G')^*)=(Is(G'))^*$. Therefore, if $G_0$ is an addition to $G$, then $G^*_0=G_0^I/D$ is an addition to $G^*=G^I/D$. Assume that $A_1$ is the image of $G^*_0$ in $H^*$ under the assumed isomorphism. By Lemma ~\ref{lem:additon-auto} there exists an automorphism of $H^*$ mapping $A$ onto $H_0^*$. Hence, $$(G_f)^*\cong G^*/G^*_0 \cong H^* /H_0^*\cong (H_f)^*$$ implying that $G_f\equiv H/H_0$. 
\end{proof}
\begin{proposition} \label{prop:GLn-main}
    Let $F$ be an NFT-field, $G=\GLF$ and $H\equiv G$. Then there exists an addition $H_0$ of $H$, that is, a subgroup $H_0\leq Z(H)$, such that $Is(H')Z(H)=Is(G')\times H_0$, a field $L\equiv F$, and an addition $K_0$ of $K=\GL_n(L)$ such that $H/H_0 \cong  K_f$.
\end{proposition}

\begin{proof} Since $G'=\SLF$ and $G'\cap G_0=1$ and $G_0\leq Z(G)$, $(G_f)'\cong \SLF$. Therefore $(G_f)'$ is regularly bi-interpretable with $F$ with respect to the images of the constants $\la{T}$ in $G_f$. We shall identify $(G_f)'$ with $\SLF$ and note that the image of $Is(G')$ in $G_f$ is also isomorphic to $Is(G')$. Consider an addition $H_0$ of $H$ as in Lemma~\ref{lem:foundation-elem}, and let $H_f=H/H_0$. We can now conclude that $(H_f)'\cong \SL_n(L)$ for some field $L\equiv F$.  Again, to keep the notation light we identify $(H_f)'$ with the copy of $SL_n(L)$ in $H_f$. From this point on, we also set $K=\GL_n(L)$ and pick an addition $K_0$ of $K$. 

Consider the first-order formula, say $\Phi(x)$, resulting from the conjunction of all the following:
$$ \exists y \in T_{ij} (xt_{ij}x^{-1}= y), \text{ if } i=1 \text{ or } j=1$$
$$ xt_{ij}x^{-1}= 1 , \text{ if } i\neq 1 \text{ and }  j\neq 1.$$
The formula defines a subgroup $\Delta_1(G_f)$. Furthermore, $G_f\models \Theta \wedge \Psi$, where
$$\Theta: \forall y ((y\in T_{12}\wedge y\neq 1) \to \exists x (\Phi(x) \wedge xt_{12}x^{-1}=y)$$
and 
$$\Psi: \forall y_1,y_2 ((\Phi(y_1)\wedge \Phi(y_2) \wedge y_1t_{12}y_1^{-1}=y_2t_{12}y_2^{-1}) \to y_1y_2^{-1} \in Z(G_f)) $$
 Hence $\Delta_1(G_f)/Z(G_f)$ is definably isomorphic to $\FM$ via an isomorphism, say $\psi_G: \Delta_1/Z(G_f) \to \FM$. The same formula $\Phi$ defines in $H$ a subgroup $\Delta_1(H)$  isomorphic to $L^\times \cdot Z(H_f)$ where $\Delta_1(H_f)/Z(H_f)$ is isomorphic to $L^\times$ via an isomorphism, say $\psi_H$ and $H_f$ is generated by $\Delta_1(H_f)$ and $(H_f)'$. Since $Z(G_f)$ is finite, $(G_f)'\cdot T(\Delta_1(F))$ is definable in $G_f$ and regularly bi-interpretable with $F$. Hence, we might identify $H_f \cdot Z(H_f)$ with the copy of $SL_n(L)\cdot Z (K_f)$ in $K_f$. 
 
 In $G=\GLF$, we have that for any $\alpha\in \FM$, Equation~\eqref{eqn:d_1-to-n} holds in $G$. Therefore, the following holds in $G_f$ 
\begin{equation} \label{eqn: xzn} G_f\models \forall x\in \Delta_1(G_f) \left((xZ)^n=\prod_{i=2}^{n} d_1(\psi_G(xZ)) d_i((\psi_G(xZ))^{-1}) Z\right)\end{equation}

The same holds for $H_f$ with obvious substitutions. Now, let $\mu:L^\times \to T(L^\times)$ be the defining transversal (homomorphism) of the addition $K_0$. Define a map $\phi:H_f \to K_f$ such that
\begin{itemize}

\item $\phi=Id$ on $H_f'\cdot Z(H_f)$
\item $\phi(\alpha)= d_1(\psi(\alpha)[\mu(\psi(\alpha))]^{-1})$, where $\alpha\in \Delta_1$, ranges over all representatives of the cosets of $\Delta_1/Z$. Note that $\Delta_1$ splits over $Z$ and $\psi$ depends only on the cosets.
\end{itemize}
For $\alpha$ as above, $\psi(\alpha) \in L^\times$ can be written as $t \cdot \beta$, where $t$ is torsion and $\beta$ is torsion-free. The element $t$ is uniquely determined. 

Let us agree that in this context $\diag$ is understood as an isomorphism $\diag: T(L^\times) \to Z(K_f)$.

Again identifying $\SL_n(L)$ with its copy in $K_f$ we can verify that
\begin{equation} \label{eqn:K_f} K_f \cong \left[\left(\SL_n(L) \cdot Z(G_f)\right)\rtimes L^\times\right] /N\end{equation}
where $L^\times$ acts on $\SL_n(L)$ as $d_1(L^\times)$ and $N$ is the normal subgroup generated by all the following elements as $\alpha$ ranges over $L^\times$:
\begin{equation} \label{eqn:K_f-relation} \left(\prod_{i=2}^{n} (d_1(\alpha) d_i(\alpha)^{-1}) \diag (t \left[\mu(\alpha))\right]^{-1}, \alpha^{-n}\right) \end{equation}

Now, since~\eqref{eqn: xzn} also holds in $H_f$ with obvious substitutions, we observe that for $\alpha\in \Delta_1$ as described, 
$$\phi^n(\alpha)= \diag\left(t[\mu(\psi(\alpha)]^{-1}\right)\prod_{i=2}^{n} d_1(\psi(\alpha)) d_i(\psi(\alpha))^{-1}$$
 Hence, we are guaranteed that $\phi$ is well defined on $\Delta_1\cap [(H_f)'\cdot Z(H_f)]$. This finishes the proof.\end{proof} 
\subsubsection{Abelian Deformations of $\GL_n$ over NFT-fields}
Let $L$ be an NFT-field, and let $H=\GL_n(L)$. Consider an addition $H_0$ and the corresponding foundation $H_f$. Identify $(H_f)'$ with $\SL_n(F)$ and let:
\begin{enumerate}
\item $\Sigma_1(H_f)=\{(\prod_{i=2}^{n} (d_1(\alpha) d_i(\alpha)^{-1})|\alpha\in L^{\times}\}$;
    \item $\Delta_1(H_f)=Is_{H_f}(\Sigma_1(H_f) \cdot Z(H_f))$ [Note that this is the same subgroup as the one defined in the proof of~\ref{prop:GLn-main}.]
    \item Let $\pi:H \to H_f$ be the canonical epimorphism and let $H_1=(H_f)'\cdot Z(H_f)$. Observe that $\pi(Is_H(H'))=H_1$. Since $Is_H(H') \cap H_0=1$, $\pi$ maps $Is_H(H')$ isomorphically onto $H_1$
\end{enumerate}

We note that $H_1\cap \Delta_1= \Sigma_1 \cdot Z(H_f)$. By~\eqref{eqn:K_f} and \eqref{eqn:K_f-relation} we might represent elements of $H_f$ as pairs $(c,a)$ where $c\in H_1$, $a\in \Delta_1$. Furthermore, there exists a symmetric 2-cocycle $p:  A/A^n \times A/A^n \to \Sigma_1\cdot Z(H_f)$, where the product in $H_f$ is given by:
$$(c_1,a_1)(c_2,a_2)=(c_1c_2^{a_1}p(a_1\Sigma_1,a_2\Sigma_2), a_1a_2).$$
   
\begin{defn} 
\label{def:ab-def-gln} 
  Let $L$ be an NFT-field, $L^\times= A(L)  \times T(L)$, where $A=A(L)$ is torsion-free and $T=T(L)$ is torsion, $H=\GL_n(L)$ and $H_f=G/G_0$ for some addition $H_0$ of $H$. Consider $B \equiv A$ and $h\in S^2(A/A^n, B)$. Define $E=GL_n(L,h,B)$ as the central extension:
$$1\to B\to E \to H_f \to 1$$
where the group multiplication is given by:
$$(b_1, c_1a_1)(b_2,c_2a_2)=(b_1b_2 h(a_1\Sigma_1,a_2\Sigma_1), c_1c_2^{a_1}p(a_1\Sigma_1, a_2\Sigma_1)a_1a_2)$$
for $b_i\in B$, $c_i\in H_1$ and $a_i\in \Delta_1(H_f)$.
\end{defn}
\begin{thm}\label{thm:NFT-gln}
Let $F$ be an NFT-field. If $H\equiv \GLF$, then $H\cong GL_n(L,h,B(H))$, as in Definition~\ref{def:ab-def-gln}, for some $L\equiv F$, $h\in S^2 (A/A^n, B)$, and $B\equiv \FM/T$.    
\end{thm}
\begin{proof}
Let $G= \GLF$. Pick any addition $G_0$ of $G$. Although it may not be definable in $G$, it is absolutely interpretable in $G$ as $Z(G)/T(Z(G))$. Therefore, in $H$ any addition $H_0$ is elementarily equivalent to $G_0$. The existence of $K=\GL_n(L)$ such that $H_f \cong K_f$ was established in Proposition~\ref{prop:GLn-main}. Therefore, $H$ fits in a central extension:
$$1\to H_0 \to H \xrightarrow{\pi} K_f \to 1.$$
We shall identify $H_f$ with $K_f$, and take $\pi: H \to H_f$ to be the canonical epimorphism. The restriction of $\pi$ to $Is(H')$ is an isomorphism. We identify $SL_n(L)\cong H'$ with its copy in $H$. Furthermore, the subgroup $\pi^{-1}(\Sigma_1(H_f))$ can also be isomorphically identified with a subgroup of $Is_H(H')$ and we denote it by $\Sigma_1(H)$. Now, $H$ is generated by $Is(H')$ and $\Pi_1(H)=Is_H(H_0 \cdot \Sigma_1(H))$. Therefore, the product in $H$ will be completely determined by the products in $Is(H')$, $\Pi_1(H)=Is_H(Z(H) \cdot \Sigma_1(H))$ and the conjugate action of the latter on the former. The conjugate action is completely determined by the structure of $H_f$, which is known up isomorphism. We fully understand the structure of $Is_H(H')$. The subgroup $\Pi_1(H)$ is an abelian group that fits into
$$ 1\to Z(H) \times \Sigma_1(H) \to \Pi_1(H) \to A/A^n\to 1.$$ 
In fact, $\Pi_1(G)$ is a definable abelian subgroup of $G$. Note that both $Z(G)$ and $\Sigma_1(G)$ are definable in $G$ with parameters $\mathcal{T}$ and $Is_G(Z(G)\cdot \Sigma_1(G))=\{ x\in G| x^n\in Z(G)\cdot \Sigma_1(G)\}$. Hence, the same formulas that define $\Pi_1(G)$ in $G$ define $\Pi_1(H)$ in $H$. Now, since
$$Ext(A/A^n, \Pi_1(H) \times H_0)\cong Ext(A/A^n, \Pi_1(H))\times T(Z(H))) \times Ext(A/A^n, H_0)$$ while the extension of $\Pi_1(H)\times T(Z(H))$ by $A/A^n$ is completely determined within $H_f$ by the 2-cocycle $p$ and the conjugate action of elements of $\Pi_1(H)$ on $H_1$. What is left to completely fix the product in $H$ is determined by a 2-cocycle $h\in S^2(A/A^n, H_0)$, which cannot necessarily be recovered from the first-order theory of $G$. \end{proof}
\subsubsection{A group elementarily equivalent to $\GL_n(\Q)$ but not isomorphic to any $\GL_n$} \label{sec:elemnotiso-gln}
Let $L$ be a countable non-standard model of the field $\Q$ and $H=\GL_n(L)$. Then by an argument similar to the one in the proof of Theorem~\ref{Z-inter-Q*:lem}, $L^\times$ contains a countable torsion-free divisible (radicable) subgroup $D$. Identify $D$ with a central subgroup of $\GL_n(L)$. Since $D$ is divisible, it will split from the center and from the group. Therefore,

$$\GL_n(\Q)\equiv \GL_n(L) \cong \GL_n(L) \times D \equiv \GL_n(\Q) \times D$$

First, $\GL_n(\Q) \ncong \GL_n(\Q) \times D$, since $Z(\GL_n(\Q))\cong \QM$ has no nontrivial divisible (radicable) subgroups. Now, assume $H= \GL_n(F) \cong \GL_n(\Q) \times D = G$ for some field $F$. Then $H'=\SL_n(F)$ and $G'=\SL_n(\Q)$. The isomorphism must map $H'$ onto $G'$, while by the abstract isomorphism theorem for $\SL_n$, $F\cong \Q$.  Contradiction!\qed

%\nocite{*}


\begin{thebibliography}{99}

\bibitem{AKNS}
M.\,Aschenbrenner, A.\,Kh\'elif, E.\,Naziazeno, T.\,Scanlon, {\it The logical complexity of finitely generated commutative rings}, Int. Math. Research Notices, {\bf 2020} (1), 2020, pp.\,112--166.
\bibitem{AvniMeiri1} Avni N., Lubotzky A.,  Meiri C. \emph{First order rigidity of non-uniform higher
rank arithmetic groups.} Invent. Math., {\bf 217}(1),  2019, 219--240.

\bibitem{AvniMeiri2} Avni N.,  Meiri C. \emph{On the model theory of higher rank arithmetic groups}, 2023, Duke Mathematical Journal, {\bf 1},
DOI:10.1215/00127094-2022-0105
\bibitem{beleg94} O. V. Belegradek, The model theory of unitriangular groups, Ann. Pure App. Logic,
68 (1994) 225-261.
\bibitem{beleg99} O. V. Belegradek, Model theory of unitriangular groups, \textit{Model theory and applications},
  Amer. Math. Soc. Transl. Ser. 2, 195, Amer. Math. Soc., Providence, RI, (1999) 1-116.
\bibitem{BM}  C. Beidar, A. Michalev,  On Malcev's theorem on elementary equivalence of
linear groups,  Contemporary mathematics,  1992,  V. 131,  P. 29-35.



\bibitem{Bunina1}
E. Bunina, Elementary equivalence of unitary linear groups over fields. Fund and Applied Math, 1998, v. 4,  4, p. 1265-1278.

\bibitem{Bunina2} E. Bunina, Elementary equivalence of unitary linear groups over rings and skew fields. Uspehi mat nauk,  1998, v  53,  2, p.  137-138.


\bibitem{Bunina3} E. Bunina, Elementary equivalence of Chevalley groups,   Uspehi Mat. Nauk, 2001, v.56, 1, p.157-158.

\bibitem{Bunina4} E. Bunina, Elementary equivalence of Chevalley groups over local rings, Mat. Sbornik, 2010, v. 201, 3, p.101-120.

\bibitem{Bunina5} E. Bunina, Isomorphisms and elementary equivalence of Chevalley groups over commutative rings, Sb. Math., 210, 8 (2019), p. 1067–1091. 

\bibitem{BG} E. Bunina,  P. Gvozdevsky, Regular bi-interpretability and finite axiomatizability of Chevalley groups, arXiv:2311.01954.

\bibitem{BMPbook} E. Bunina, A. Mikhalev, A. Pinus, Elementary equivalence and close to it logical equivalences of classical and universal algebras. 

\bibitem{BMP} E. Bunina, A. Myasnikov, E. Plotkin,  The Diophantine problem in Chevalley groups,  J. of Algebra, 2024, https://doi.org/10.1016/j.jalgebra.2024.04.003.


\bibitem{CK} D. Carter,  G. Keller  Bounded Elementary Generation of $SL_n(\mathcal{O})$, Am. J. Math. Vol. 105, No. 3, 1983, pp. 673-687.
%\bibitem{Casals2}  M. Casals-Ruiz, I. Kazachkov,  
%On systems of equations over free products of groups,
%Journal of Algebra, 333, 1, 2011, 368 - 426. 

\bibitem{DM1} E. Danyarova, A. Myasnikov, Groups elementarily equivalent to metabelian Baumslag-Solitar groups and regular bi-interpretability, Math Arxiv.


\bibitem{eklof}P. C. Eklof, and R. F. Fischer, The elementary theory of abelian groups, Ann. Math.  Logic, 4(2) (1972)  115-171.

\bibitem{DV} Erovenko, Igor V. $\text{SL}_n(F[x])$ is not boundedly generated by elementary matrices: explicit proof.(English summary)vElectron. J. Linear Algebra (2004), 162–167.
\bibitem{fox} L. Fuchs, \textit{Infinite Abelian Groups}. Vol. 1. Academic Press, New York, 1970.
\bibitem{GMO}   A. Garreta, A. Miasnikov, D. Ovchinnikov,  Diophantine problems in solvable groups, Bulletin of Mathematical Sciences, 
Vol. 10, No. 1 (2020), 27 pages.
DOI: 10.1142/S1664360720500058


\bibitem{GMO_rings}
  A. Garreta, A. Miasnikov, D. Ovchinnikov,
    {Diophantine problems in rings and algebras: undecidability and reductions to rings of algebraic integers},  		arXiv:1805.02573 [math.RA]
  
 \bibitem{GMO_comm}
  A. Garreta, A. Miasnikov, D. Ovchinnikov, 	Diophantine problems in commutative rings, arXiv:2012.09787 [math.NT].

\bibitem{KM} M. Kargapolov, Yu. Merzljakov, Fundamentals of the Theory of Groups,  Graduate Texts in Mathematics, 1979, Springer.

\bibitem{Khelif}
A.\,Kh\'elif, {\it Bi-interpr\'etabilit\'e et structures QFA: \'etude de groupes r\'esolubles et des anneaux commutatifs}, Comptes Rendus. Math., {\bf 345}, 2007, pp.\,59--61.

\bibitem{KMS} O. Kharlampovich, A. Myasnikov, M. Sohrabi, {\it Rich groups, weak second order logic, and applications}, in Groups and Model Theory, GAGTA BOOK 2, DE GRUYTER, 2021.
\bibitem{KM2}   O. Kharlampovich, A. Myasnikov,  {\it Tarski-type problems for free associative algebras},   J. Algebra, Volume 500, 2018, 589-643.
  

\bibitem{KM3}  O. Kharlampovich, A. Myasnikov,  {\it What does a group algebra of a free group "know" about the group?}  Annals of Pure and Applied Logic, 169 (6), 523-547 (2018).
  
\bibitem{Mal} A. I. Malcev. {On isomorphic matrix representations of infinite groups}.
Rec. Math. [Mat. Sbornik] {\bf 8}(50), 1940, 405--422 (Russian. English summary).

\bibitem{M87}  Myasnikov A. {Elementary theories and abstract isomorphisms of finite-dimensional algebras and unipotent
groups}. Dokl. Akad. Nauk SSSR, {\bf 297}(2), 1987,  290--293.

\bibitem{MS1} A. G. Myasnikov, M. Sohrabi,  On  groups elementarily equivalent to a group of triangular matrices $T_n(R)$, arXiv:1609.09802
 
\bibitem{MS2} A. G. Myasnikov, M. Sohrabi, { Complete first-order theories of the classical matrix groups over algebraic integers}, J. of Algebra, 582, 2021, 206-231.

\bibitem{MS2012} A. G. Myasnikov, M. Sohrabi, Elementary coordinatization of finitely generated nilpotent groups, arXiv:1311.1391



\bibitem{MS3} A. G. Myasnikov, M. Sohrabi, The Diophantine problem in the classical matrix groups.  Izv. Math., 2021, 85, DOI:10.1070/IM9104.

\bibitem{MN1} A. G. Myasnikov, A. Nikolaev, Nonstandard polynomials: algebraic properties and elementary equivalence, Math. Arxiv.

\bibitem{Nies}  
A.\,Nies, {\it Comparing quasi-finitely axiomatizable and prime groups}, J. Group Theory, {\bf 10} (2007), pp.\,347--361.


\bibitem{julia} J. Robinson, The undecidability of algebraic rings and fields, Proc. Amer. Math. Soc. 10 (1959) 950-957.

\bibitem{Segal-Tent} Segal D., Tent K. \emph{Defining $R$ and $G(R)$}. Journal of the European Mathematical Society, 2020, {\bf 25}(8), 3325--3358.

\bibitem{Kallen} W. van der Kallen, $\text{SL}_3(\mathbb{C}[x])$ does not have bounded word length, Lecture Notes in Math. 966 (1982), 357–361
\end{thebibliography}
\end{document}